\documentclass[11pt,twoside]{article}
\usepackage{fancyhdr,color,graphicx,amsmath,amsfonts,amssymb,latexsym,bm,indentfirst,cite,amsthm,enumerate}
\usepackage[colorlinks=true,backref=page]{hyperref}
\usepackage[all]{xy}
\usepackage[a4paper,left=25mm,right=25mm,top=35mm,body={155mm,230mm}]{geometry}
\usepackage{pdfpages}
\usepackage{titletoc}
\usepackage{fancyhdr}
\newtheorem*{claim}{\hspace{2em} Claim}

\newtheorem{Theorem}{Theorem}[section]
\newtheorem{Definition}[Theorem]{ Definition}
\newtheorem{Lemma}[Theorem]{ Lemma}

\newtheorem{Proposition}[Theorem]{ Proposition}

\newtheorem{theoremalph}{Theorem}

\newtheorem*{Remark}{Remark}

\makeatletter
\newcommand{\subsectionruninhead}{\@startsection{subsection}{2}{0mm}{-\baselineskip}{-0mm}{\bf\large}}
\newcommand{\subsubsectionruninhead}{\@startsection{subsubsection}{3}{0mm}{-\baselineskip}{-0mm}{\bf\normalsize}}
\makeatother

\makeatletter
\newsavebox{\@brx}
\newcommand{\llangle}[1][]{\savebox{\@brx}{\(\m@th{#1\langle}\)}
	\mathopen{\copy\@brx\kern-0.5\wd\@brx\usebox{\@brx}}}
\newcommand{\rrangle}[1][]{\savebox{\@brx}{\(\m@th{#1\rangle}\)}
	\mathclose{\copy\@brx\kern-0.5\wd\@brx\usebox{\@brx}}}
\makeatother

\begin{document}
   \newpage
   \title{The Lyapunov Exponents of Hyperbolic Measures for $C^1$ Vector Fields with Dominated Splitting}
   \author{ Wanlou Wu \footnote{Wanlou Wu was supported by NSFC 12001245, 12271260, 12471185.}}
   \date{}
   \maketitle

\begin{abstract}
   In this paper, we prove that for every $C^1$ vector field preserving an ergodic hyperbolic invariant measure which is not supported on singularities, if the Oseledec splitting of the ergodic hyperbolic invariant measure is a dominated splitting, then the ergodic hyperbolic invariant measure can be approximated by periodic measures, and the Lyapunov exponents of the ergodic hyperbolic invariant measure can also be approximated by the Lyapunov exponents of those periodic measures. 
\end{abstract}

\section{Introduction}
   Let $M^d$ be a compact $d$-dimensional $C^\infty$ Riemannian manifold without boundary. We will simply denote it by $M$ when no confusion arises. Denote by $\mathfrak{X}^r(M)(r\geq 1)$ the space of all $C^r$ vector fields on $M$. Let $\varphi^X_t=\{\varphi_t\}_{t\in\mathbb{R}}$ be the $C^1$ flow generated by a vector field $X\in\mathfrak{X}^1(M)$, for simplicity, we write it as $\varphi_t$. The derivative of the flow with respect to the space variable is called the \emph{\bf tangent flow} and is denoted by $\Phi_t={\rm d}\varphi_t$. Fix a smooth Riemannian metric on $M$. This induces a scalar product on each tangent space $T_xM$ that varies differentially with $x$. The limit $$\chi(x,v)=\lim_{t\rightarrow\pm\infty}\dfrac{1}{ t}\log\lVert\Phi_t(v)\rVert,~v\in T_xM,~v\neq0,~x\in M,$$ is called a \emph{\bf Lyapunov exponent} for a tangent vector $v\in T_xM$. The Lyapunov exponents for a differential equation are natural generalizations of the eigenvalues of the matrix in the linear part of the equation. The Lyapunov exponents describe asymptotic evolution of a tangent map: positive (resp. negative) exponents correspond to exponential growth (resp. decay) of the norm, while vanishing exponents indicate the absence of exponential behavior. 
    
   By the Oseledec Theorem \cite{OV}, the limit $\chi(x,v)$ exists for all nonzero vectors $v$ based on almost all state points $x\in M$ with respect to each given invariant measure, and it is independent of the points if the measure is ergodic. Precisely, let $\mu$ be an invariant measure of flow $\varphi_t$. By the Oseledec Theorem \cite{OV}, for $\mu$-almost every $x\in M$, there exist a positive integer $k(x)$, real numbers $\chi_1(x)<\chi_2(x)<\cdots<\chi_{k(x)}(x)$ and a measurable $\Phi_t$-invariant splitting $$T_xM=E_1(x)\oplus E_2(x)\oplus\cdots\oplus E_k(x)$$ such that $$\lim_{t\to\pm\infty}\dfrac{1}{ t}\log\lVert\Phi_t(v)\rVert=\chi_i(x),~\forall~v\in E_i(x),~v\neq 0,~i=1,2,\cdots,k(x).$$ The numbers $\chi_1(x),\cdots,\chi_{k(x)}(x)$ are called \emph{\bf Lyapunov exponents} at point $x$ of $\Phi_t$ with respect to $\mu$. Denote by $$d_i=\text{dim}(E_i(x)),\quad i=1,2,\dots,k(x)$$ the multiplicities of those Lyapunov exponents, and the vector formed by these numbers (counted with multiplicity, endowed with the increasing order) is called \emph{\bf Lyapunov vector} at point $x$ of $\Phi_t$ with respect to $\mu$. Disregarding the splitting, we list the Lyapunov exponents of $\mu$ (with multiplicity) as $\lambda_1(x)\leq\lambda_2(x)\leq\cdots\leq\lambda_d(x)$. That is, $$\lambda_j(x)=\chi_i(x),\quad\text{ for every } d_1+d_2+\cdots+d_{i-1}<j\leq d_1+d_2+\cdots+d_i.$$ If $\mu$ is ergodic, then these numbers $k(x),\chi_1(x),\cdots,\chi_{k(x)}(x)$ are constants independent of $x$. The Lyapunov exponent is an active topic in the theory of nonuniformly hyperbolic systems, known as Pesin theory, which recovers hyperbolic behavior for the points whose Lyapunov exponents are all nonzero. For such points, there exist well-defined unstable and stable invariant manifolds. By the Poincar\'{e} Recurrence Theorem, for $\mu$-almost every $x\in M$, $$\lim_{t\to\pm\infty}\dfrac{1}{t}\log\parallel\Phi_t|_{\langle X(x)\rangle}\parallel=0,$$ where $\langle X(x)\rangle$ denotes the $1$-dimensional subspace of $T_xM$ spanned by the flow direction $X(x)$. Consequently, an ergodic invariant measure is called \emph{\bf hyperbolic} if all its Lyapunov exponents are nonzero except the one corresponding to the flow direction.    

   For diffeomorphisms, Anosov \cite{A67} proved that a non-wandering orbit segment has a periodic orbit nearby for uniformly hyperbolic diffeomorphisms. For Axiom A diffeomorphisms, Sigmund \cite{Sig70} proved that periodic measures are dense in the set of all invariant measures. Katok \cite{Ka} showed that hyperbolic periodic points are dense in the closure of the basin of a given hyperbolic measure for $C^{1+\alpha}(\alpha>0)$ diffeomorphisms. Later, Wang and Sun \cite{WW10} enhanced Katok's work by showing that the Lyapunov exponents of an ergodic hyperbolic invariant measure for $C^{1+\alpha}(\alpha>0)$ diffeomorphisms can be approximated by those of the atomic measures supported on hyperbolic periodic orbits (see also \cite{ZS10} for $C^1$ diffeomorphisms when the Oseledec splitting is dominated). Further approximation results can be found in \cite{LLS09,LLS14}. Since a vector field generates a continuous-time dynamical system and a diffeomorphism can be viewed as the time-$1$ map of a suitable vector field, the dynamic behavior of vector fields and diffeomorphisms are similar in most cases. By considering the sectional Poincar\'{e} maps of the vector field, sometimes one can get some (not all) similar properties between $d$-dimensional vector fields and $(d-1)$-dimensional diffeomorphisms. It is natural to ask the following questions:    
\begin{enumerate}[]
   \item \textcolor{red}{$Q_1$}: Are hyperbolic periodic points dense in the closure of the basin of a given hyperbolic measure for smooth vector fields?
   
   \item \textcolor{red}{$Q_2$}: Can the Lyapunov exponents of an ergodic hyperbolic invariant measure for smooth vector fields be approximated by those of the atomic measures supported on hyperbolic periodic orbits?   
\end{enumerate}        
   
   For the question \textcolor{red}{$Q_1$}, Ma \cite{M22} stated that hyperbolic periodic points are dense in the closure of the basin of a given hyperbolic measure for smooth semiflows on separable Banach spaces. Recently, Li, Liang and Liu \cite{LLL24} proved that for $C^{1+\alpha}(\alpha>0)$ nonuniform hyperbolic vector fields, every ergodic hyperbolic invariant measure which is not supported on singularities can be approximated by periodic measures. Lu and Wu \cite{LW2025} gave positive answers to both questions for $C^1$ star vector fields on three-dimensional manifolds. Compared with diffeomorphisms, vector fields present additional challenges due to the presence of singularities. Flows with singularities exhibit rich and complicated dynamics, a famous example being the \emph{Lorenz attractor} \cite{Lor63,Guc76}. At singularities, one cannot define the \emph{linear Poincar\'{e} flow} (see Definition \ref{Def:linearPoincare}). Hence we lose some compact properties. This prevents one from directly applying certain techniques developed for diffeomorphisms, such as Crovisier's central model \cite{CS10} and the distortion arguments of Pujals-Sambarino \cite{PS00}, to singular vector fields. Although the sectional Poincar\'{e} maps of the vector field can get similar properties as diffeomorphisms with lower one dimensional, some vector fields (the famous Lorenz attractor \cite{Lor63,Guc76}) displays different dynamics. In the spirit of the Lorenz attractor, geometric Lorenz attractors \cite{ABS77,Guc76,GW79} were constructed in a theoretical way. Loosely speaking, a geometric Lorenz attractor is a robust attractor and it contains a hyperbolic singularity which is accumulated by hyperbolic periodic orbits in a robust way. The return map of a geometric Lorenz attractor is discontinuous. This creates additional obstacles when one wants to generalize Ma\~{n}\'{e} \cite{MR} classical argument and Katok's \cite{Ka} classical result. In this paper, we establish a relation between the Lyapunov exponents of an ergodic hyperbolic invariant measure for $C^1$ vector fields on a $d$-dimensional ($d\geq 3$) manifold and those of hyperbolic periodic measures. Our main results are stated in Theorems \ref{ThmA} and \ref{ThmB}.     	
\begin{theoremalph}\label{ThmA}
   Let $\varphi_t$ be the $C^1$ flow generated by a vector field $X\in\mathfrak{X}^1(M)$ on a compact $d$-dimensional Riemannian manifold $M$, and let $\mu$ be an ergodic hyperbolic invariant measure which is not supported on any singularity. If the Oseledec splitting of $\mu$ is a dominated splitting {\rm (}see Definition \ref{OSDS}{\rm )}, then there exists a sequence of periodic measures converging to $\mu$ in the weak$^*$ topology.   
\end{theoremalph}

   Regarding Question  \textcolor{red}{$Q_2$}, results are so far available only for $3$-dimensional $C^1$ star vector fields, due to Lu and Wu \cite{LW2025}. In the present work, building on Theorem \ref{ThmA}, we establish that the Lyapunov exponents of every ergodic hyperbolic invariant measure can be approximated by those of hyperbolic periodic measures. The Lyapunov exponents of a periodic measure concentrated on a periodic orbit with period $\pi$ are exactly the logarithm of the norms of the absolute values of eigenvalues of $\Phi_\pi$. By the Poincar\'{e} Recurrence Theorem, the Lyapunov exponent corresponding to the flow direction is always zero for the tangent flow $\Phi_t$. Therefore, we restrict our attention to the Lyapunov exponents of the (scaled) linear Poincar\'{e} flow (see Deﬁnition \ref{Def:linearPoincare}) with respect to the ergodic hyperbolic invariant measure.    
\begin{theoremalph}\label{ThmB}
   Let $\varphi_t$ be the $C^1$ flow generated by a vector field $X\in\mathfrak{X}^1(M)$ on a compact $d$-dimensional Riemannian manifold $M$, and let $\mu$ be an ergodic hyperbolic invariant measure which is not supported on any singularity, with Lyapunov exponents $\lambda_1\leq\lambda_2\leq\cdots\leq\lambda_{d-1}$ of the {\rm (}scaled{\rm )} linear Poincar\'{e} flow. If the Oseledec splitting of the {\rm (}scaled{\rm )} linear Poincar\'{e} flow with respect to $\mu$ is a dominated splitting {\rm (}see Definition \ref{OSDS}{\rm )}, then the Lyapunov exponents of $\mu$ can be approximated by the Lyapunov exponents of hyperbolic periodic measures. To be precise, for every $\varepsilon>0$, there exists a hyperbolic periodic point $p$ with Lyapunov exponents $\lambda_1(p)\leq\lambda_2(p)\leq\cdots\leq\lambda_{d-1}(p)$ such that $$\lvert\lambda_i-\lambda_i(p)\rvert<\varepsilon,~i=1,2,\cdots,d-1.$$  
\end{theoremalph}
  
   One of the main difficulties in proving Theorems \ref{ThmA} and \ref{ThmB} is the presence of singularities. Even in the absence of singularities, we cannot directly apply the standard Pesin theory as in Lian and Young \cite{Lian} because the vector field is merely $C^1$. It is well known that $C^1$ systems possess properties fundamentally different from those of $C^{1+\alpha}(\alpha>0)$ systems. For instance, in $C^1$ nonuniform hyperbolic systems, the stable and unstable manifolds may fail to exist in two distinct ways: they may not exist at all (see \cite{PC84}), or they may exist but fail to be absolutely continuous (see \cite{SW00}). The main challenge is to overcome the difficulties caused by singularities and $C^1$ differentiability. Liao's shadowing lemma for singular flows \cite{LST85} furnishes a technique for identifying periodic points. Nevertheless, unlike the situation for diffeomorphisms, we are unable to use it to derive estimates for the Lyapunov exponents corresponding to periodic measures of the flow. The outline of the paper is as follows. Section \ref{P} covers the necessary background on vector fields. Section \ref{LP} details the Lyapunov metric and the shadowing lemma. Sections \ref{MA} and \ref{PF} are devoted to proving Theorems \ref{ThmA} and \ref{ThmB}, respectively.   
        
\section{Preliminaries}\label{P}	

\subsection{Basic Contents of Vector Fields}
   Let $M^d$ be a compact $d$-dimensional $C^\infty$ Riemannian manifold without boundary, we will write it simply $M$ when no confusion can arise. Denote by $\mathfrak{X}^r(M)(r\geq 1)$ the space of all $C^r$ vector fields on $M$. Given $X\in\mathfrak{X}^1(M)$, a point $x\in M$ is called a \emph{singularity} if $X(x)=0$. Denote by ${\rm Sing}(X)$ the set of all singularities. A point $x$ is \emph{regular} if $X(x)\neq 0$. Let $\varphi^X_t=(\varphi_t)_{t\in\mathbb{R}}$ be the $C^1$ flow generated by a vector field $X$, for simplicity, denoted by $\varphi_t$. A regular point $p$ is \emph{periodic}, if $\varphi_{t_0}(p)=p$ for some $t_0>0$. A critical point is either a singularity or a periodic point. Denote the \emph{normal bundle} of $X$ by $$\mathcal{N}\triangleq\bigcup\limits_{x\in M\setminus{\rm Sing}(X)}\mathcal{N}_x,$$ where $\mathcal{N}_{x}$ is the orthogonal complement space of the flow direction $X(x)$, i.e., $$\mathcal{N}_{x}=\{v\in T_{x}M: v\perp X(x)\}.$$ For the flow $\varphi_t$ generated by $X$, its derivative with respect to the space variable is called the \emph{tangent flow} and is denoted by $\Phi_t={\rm d}\varphi_t$.
     
\begin{Definition}\label{Def:linearPoincare}
   Given $x\in M\setminus{\rm Sing}(X),v\in\mathcal{N}_x$ and $t\in\mathbb{R}$, the \emph{\bf linear Poincar\'{e} flow} $$\psi_t:\mathcal{N}\to\mathcal{N}$$ is defined as $$\psi_t(v)\triangleq\Phi_t(v)-\frac{\langle\Phi_t(v),X(\varphi_t(x)) \rangle}{\left\lVert X(\varphi_t(x))\right\rVert^2} X(\varphi_t(x)).$$ Namely, $\psi_t(v)$ is the orthogonal projection of $\Phi_t(v)$ on $\mathcal{N}_{\varphi_t(x)}$ along the flow direction $X(\varphi_t(x))$.  
\end{Definition}  
     
   Fix $T>0$, the norm $$\lVert\psi_T\rVert=\sup\left\{\lvert\psi_T(v)\rvert:~v\in\mathcal{N},~\lvert v\rvert=1\right\}$$ is uniformly upper bounded on $\mathcal{N}$, although $M\setminus{\rm Sing}(X)$ may be not compact. Denote by $$m(\psi_T)=\inf\left\{\lvert\psi_T(v)\rvert:~v\in\mathcal{N},~\lvert v\rvert=1\right\}$$ the mininorm of $\psi_T$. Since $$m(\psi_T)=\lVert\psi_T^{-1}\rVert^{-1}=\lVert\psi_{-T}\rVert^{-1},$$ the mininorm $m(\psi_T)$ is uniformly bounded away from $0$ on $\mathcal{N}$. Another useful flow is the \textbf{ scaled linear Poincar\'{e} flow} $\psi_t^*:\mathcal{N}\to\mathcal{N}$, which is defined as $$\psi^*_t(v)\triangleq\frac{\left\lVert X(x)\right\rVert}{\left\lVert X(\varphi_t(x))\right\rVert}\psi_t(v)=\frac{\psi_t(v)}{\left\lVert\Phi_t|_{\langle X(x)\rangle}\right\rVert},$$ where $x\in M\setminus{\rm Sing}(X),~v\in\mathcal{N}_x$ and $\langle X(x)\rangle$ is the $1$-dimensional subspace of $T_xM$ generated by the flow direction $X(x)$.
   
   The linear Poincar\'{e} flow $\psi_t$ loses the compactness due to the existence of singularities. To overcome this difficulty, the linear Poincar\'{e} flow can also be defined in a more general way by Liao \cite{Liao89}. Li, Gan and Wen \cite{LGW05} used the terminology of \textquotedblleft extended linear Poincar\'{e} flow\textquotedblright. For every point $x\in M$, the sphere fiber at $x$ is defined as $$S_xM=\{v:~v\in T_xM,~\lvert v\rvert=1\}.$$ Then, the sphere bundle $$SM=\bigcup_{x\in M}S_xM$$ is compact. For every $v\in SM$, one can define the \emph{\bf unit tangent flow} $$\Phi_t^I: SM\rightarrow SM$$ as $$\Phi_t^I(v)=\dfrac{\Phi_t(v)}{\lvert\Phi_t(v)\rvert}.$$ Given a compact invariant set $\Lambda$ of the flow $\varphi_t$, denote by $$\widetilde{\Lambda}=\text{Closure}\left(\bigcup_{x\in\Lambda\setminus\text{Sing}(X)}\dfrac{X(x)}{\lvert X(x)\rvert}\right)$$ in $SM$. Thus the essential difference between $\widetilde{\Lambda}$ and $\Lambda$ is on the singularities. We can get more information on $\widetilde{\Lambda}$: it tells us how regular points in $\Lambda$ accumulate singularities. For each $x\in M$, and any two orthogonal vectors $v_1\in S_xM$, $v_2\in T_xM$, one can define $$\Theta_t(v_1,v_2)=\left(\Phi_t(v_1),\Phi_t(v_2)-\frac{\langle\Phi_t(v_1),\Phi_t(v_2)\rangle}{\left\lVert\Phi_t(v_1)\right\rVert^2} \Phi_t(v_1)\right).$$ By the definition, the two components of $\Theta_t$ are still orthogonal. If we denote $$\Theta_t=\left(\text{Proj}_1(\Theta_t),\text{Proj}_2(\Theta_t)\right),$$ then for each regular point $x\in M$ and every vector $v\in \mathcal{N}_x$, one has that $$\psi_t(v)=\text{Proj}_2(\Theta_t)(X(x),v).$$ By the continuity of $\Theta_t$, one can extend the definition of $\psi_t$ to singularities: for every vector $u\in\widetilde{\Lambda}$, one can define $$\widetilde{\mathcal{N}}_u=\{v\in T_{\pi(u)}M:~v\bot u\}.$$ Then, $\widetilde{\mathcal{N}}$ is a $(d-1)$-dimensional vector bundle on the base space $\widetilde{\Lambda}$. For every $u\in\widetilde{\Lambda}$, and every $v\in\widetilde{\mathcal{N}}_u$, one can define a flow $$\widetilde{\psi}_t(v)=\text{Proj}_2(\Theta_t)(u,v).$$ Then, the linear Poincar\'{e} flow $\psi_t$ can be \textquotedblleft embedded\textquotedblright in the flow $\Theta_t$. By the definition, $\text{Proj}_2(\Theta_t)$ is a continuous flow defined on $\widetilde{\Lambda}$. Thus, $\widetilde{\psi}_t(v)$ varies continuously with respect to the vector field $X$, the time $t$ and the vector $v$ and can be viewed as a compactification of $\psi_t$.   
   
   For every $x\in M\setminus{\rm Sing}(X)$ and each sufficiently small $\delta>0$, the \textbf{ normal manifold} of $x$ is defined as $$N_x(\delta)=\text{exp}_x\left(\mathcal{N}_x(\delta)\right),$$ where $\mathcal{N}_x(\delta)=\left\{v\in\mathcal{N}_x:~\lvert v\rvert\leq\delta\right\}$. For sufficiently small $\delta>0$, $N_x(\delta)$ is an embedded submanifold which is diffeomorphic to $\mathcal{N}_x$. Furthermore, $N_x(\delta)$ is a local cross section transverse to the flow. To study the dynamics in a small neighborhood of a periodic orbit of a vector field, Poincar\'{e} introduced the sectional return map of a cross section of a periodic point. By generalizing this idea to every regular point, one can define the sectional Poincar\'{e} map between any two cross sections at any two points in the same regular orbit. For every $T>0$ and $x\in M\setminus{\rm Sing}(X)$, the flow $\varphi_t$ induces a local holonomy map, called the \textbf{ Poincar\'{e} map} $$P_{x,\varphi_T(x)}:N_x(\delta)\rightarrow N_{\varphi_T(x)}(\delta'),$$ where $\delta$ and $\delta'$ depend on the choice of $x$ and $T$. The \textbf{ sectional Poincar\'{e} map} is then defined as the lift of Poincar\'{e} map $P_{x,\varphi_T(x)}$ to the normal bundle: $$\mathcal{P}_{x,\varphi_T(x)}={\rm exp}^{-1}_{\varphi_T(x)}\circ P_{x,\varphi_T(x)}\circ {\rm exp}_x:~\mathcal{N}_x(\delta)\rightarrow\mathcal{N}_{\varphi_T(x)}(\delta').$$ To eliminate the influence of flow speed, we also consider the \emph{\bf scaled sectional Poincar\'{e} map} $\mathcal{P}^*_{x,\varphi_T(x)}$ which is defined by $$\mathcal{P}^*_{x,\varphi_T(x)}=\dfrac{\mathcal{P}_{x,\varphi_T(x)}}{\lvert X(\varphi_T(x))\rvert}.$$ Moreover, one has $$D_x\mathcal{P}_{x,\varphi_T(x)}=\psi_T|_{\mathcal{N}_x}:\mathcal{N}_x\rightarrow\mathcal{N}_{\varphi_T(x)},\quad D_x\mathcal{P}^*_{x,\varphi_T(x)}=\psi^*_T|_{\mathcal{N}_x}:\mathcal{N}_x\rightarrow\mathcal{N}_{\varphi_T(x)}.$$ 
  
   For the sectional Poincar\'{e} map, we have the following three lemmas \cite{GY,WYZ} which respectively explain the rationality of the definition, establish the uniform continuity of Poincar\'{e} map up to flow speed, and provide an estimate for the return time.    
\begin{Lemma}\label{well-def}{\rm(\cite[Lemma 2.3]{GY})} 
   Given $X\in\mathfrak{X}^1(M)$ and $T>0$, there exists $\beta_T>0$ such that for every $x\in M\setminus{\rm Sing}(X)$, the Poincar\'{e} map $$\mathcal{P}_{x,\varphi_T(x)}:N_x\left(\beta_T\lvert X(x)\rvert\right)\to N_{\varphi_T(x)}\left(\beta_TC_T\lvert X(\varphi_T(x))\rvert\right)$$ is well defined, where $C_T=\lVert\psi_T\rVert$ is a bounded constant. 
\end{Lemma}     
     
\begin{Lemma}\label{uni-con}{\rm(\cite[Lemma 2.4]{GY})} 
   Given $X\in\mathfrak{X}^1(M)$ and $T>0$, and reducing $\beta_T>0$ in Lemma \ref{well-def} if necessary, for each $x\in M\setminus{\rm Sing}(X)$, for the Poincar\'{e} map $$\mathcal{P}_{x,\varphi_T(x)}:N_x\left(\beta_T\left\lvert X(x)\right\rvert\right)\to N_{\varphi_T(x)}\left(\beta_TC_T\left\lvert X(\varphi_T(x))\right\rvert\right),$$ $D\mathcal{P}_{x,\varphi_T(x)}$ is uniformly continuous in the following sense: for every $\epsilon>0$, there exists $\rho\in(0,\beta_T]$ such that for every $x\in M\setminus{\rm Sing}(X)$ and $y,~y'\in N_x\left(\delta_T\left\lvert X(x)\right\rvert\right)$, if $\left\lvert y-y'\right\rvert\leq\rho\lvert X(x)\rvert$, then $$\left\lvert D_y\mathcal{P}_{x,\varphi_T(x)}-D_{y'}\mathcal{P}_{x,\varphi_T(x)}\right\rvert<\epsilon.$$  	
\end{Lemma}           
     
\begin{Lemma}\label{return}{\rm(\cite[Lemma 4.5]{WYZ})} 
   Given $X\in\mathfrak{X}^1(M)$, let $\beta=\beta_1$ and $C_1$ be the constants in Lemma \ref{well-def} associated to the time-$1$ map $\varphi_1$. Then by reducing $\beta>0$ if necessary, there exists $\kappa>0$ such that for every $x\in M\setminus{\rm Sing}(X)$ and $y\in N_x\left(\beta\left\lvert X(x)\right\rvert\right)$, there is a unique $t=t(y)\in(0,2)$ satisfying that $$\varphi_t(y)\in N_{\varphi_1(x)}\left(\beta C_1\left\lvert X(\varphi_1(x))\right\rvert\right)\quad\text{and}\quad\lvert t-1\rvert<\kappa d(x,y).$$   	
\end{Lemma}

   We now review fundamental properties regarding various splittings for vector fields. Similar discussions can be found in \cite{LGW05,SGW14}.        
\begin{Definition}\label{DDS}
   Let $\Lambda$ be a compact invariant set of the vector field $X$, and let $T_\Lambda M=E\oplus F$ be an invariant splitting with respect to the tangent flow $\Phi_t$ over $\Lambda$. We call the splitting $T_\Lambda M=E\oplus F$ a \textbf{dominated splitting} with respect to the tangent flow $\Phi_t$, if there exist constants $C\geq 1$ and $\lambda>0$ such that for every $x\in\Lambda$ and $t\geq 0$, the following holds $$\left\lVert\Phi_t|_{E(x)}\right\rVert\cdot \left\lVert\Phi_{-t}|_{F\left(\varphi_t(x)\right)}\right\rVert\leq Ce^{-\lambda t}.$$
\end{Definition}

\begin{Definition}
   Let $\Lambda$ be a compact invariant set of the vector field $X$. Two subbundles $\mathcal{E}$ and $\mathcal{F}$ of $T_\Lambda M$ over $\Lambda$ are dominated with respect to the tangent flow $\Phi_t$, if there exist constants $C\geq 1$ and $\lambda>0$ such that for every $x\in\Lambda$ and $t\geq 0$, the following holds $$\left\lVert\Phi_t|_{\mathcal{E}(x)}\right\rVert\cdot \left\lVert\Phi_{-t}|_{\mathcal{F}\left(\varphi_t(x)\right)}\right\rVert\leq Ce^{-\lambda t}.$$ For simplicity, denote by $\mathcal{E}\prec\mathcal{F}$. 
\end{Definition}         

\begin{Definition}\label{FHS}
   Let $\Lambda$ be an invariant set of the vector field $X$, and let $T_\Lambda M=E\oplus\langle X\rangle\oplus F$ be an invariant splitting with respect to the tangent flow $\Phi_t$ over $\Lambda$. The splitting $T_\Lambda M=E\oplus\langle X\rangle\oplus F$ is called a \textbf{hyperbolic splitting} with respect to the tangent flow $\Phi_t$, if there exist constants $C\geq 1$ and $\lambda>0$ such that for every $x\in\Lambda$ and every $t\geq 0$, $$\left\lVert\Phi_t|_{E(x)}\right\rVert\leq Ce^{-\lambda t}\text{ and }\left\lVert\Phi_{-t}|_{F(x)}\right\rVert\leq Ce^{-\lambda t},$$ where $\langle X\rangle$ denotes the $1$-dimensional subspace spanned by the flow direction. Correspondingly, the set $\Lambda$ is called hyperbolic. 
\end{Definition} 

   Note that $\langle X\rangle$ is the trivial subspace $\{0\}$ at every singularity $x\in\text{Sing}(X)$, and is a $1$-dimensional subspace at every regular point $x\in M$. Since the constants $C$ and $\lambda$ in Definition \ref{FHS} are independent of $x\in\Lambda$, it follows that within every hyperbolic set, regular points cannot accumulate at singularities. Furthermore, one can show that the hyperbolic splitting $E(x)\oplus\langle X(x)\rangle\oplus F(x)$ varies continuously in $x$. Consequently, if $\Lambda$ is hyperbolic, then its closure $\overline{\Lambda}$ is also hyperbolic. 
\begin{Remark}
   Since the linear Poincar\'{e} flow $\psi_t$, the scaled linear Poincar\'{e} flow $\psi^*_t$, and the extended linear Poincar\'{e} flow $\widetilde{\psi}_t$ are all derived from the tangent flow $\Phi_t$, one can analogously define the notions of dominated and hyperbolic splittings for each of them. For instance, let $\Lambda$ be an {\rm (}not necessarily compact{\rm)} invariant set of flow $\varphi_t$. An invariant splitting $\mathcal{N}_\Lambda=E\oplus F$ over the invariant set $\Lambda$ is called a \textbf{dominated splitting} with respect to the linear Poincar\'{e} flow $\psi_t$, if there exist constants $C\geq 1$ and $\lambda>0$ such that for every $x\in\Lambda$ and every $t\geq 0$, the following holds $$\left\lVert\psi_t|_{E(x)} \right\rVert\cdot\left\lVert\psi_{-t}|_{F\left(\varphi_t(x)\right)}\right\rVert\leq Ce^{-\lambda t}.$$ An invariant splitting $\mathcal{N}_\Lambda=E\oplus F$ over the invariant set $\Lambda$ is called a \textbf{hyperbolic splitting} with respect to the linear Poincar\'{e} flow $\psi_t$, if there exist constants $C\geq 1$ and $\lambda>0$ such that for every $x\in\Lambda$ and every $t\geq 0$, the following holds $$\left\lVert\psi_t|_{E(x)}\right\rVert\leq Ce^{-\lambda t}\text{ and }\left\lVert\psi_{-t}|_{F(x)}\right\rVert\leq Ce^{-\lambda t}.$$ Furthermore, a splitting $\mathcal{N}_\Lambda=E\oplus F$ is a {\bf dominated {\rm (}resp. hyperbolic{\rm)} splitting} with respect to $\psi_t$ if and only if it is a {\bf dominated {\rm (}resp. hyperbolic{\rm)} splitting} with respect to $\psi^*_t$ or $\widetilde{\psi}_t$. It should be emphasized that this definition also applies to two subbundles of $\mathcal{N}$ over $\Lambda$.     
\end{Remark}

\subsection{The Lyapunov Exponents of Vector Fields}\label{LEF} 
   In this section, we introduce some ergodic theory for vector fields. Given a vector field $X\in\mathfrak{X}^r(M)(r\geq1)$, a measure $\mu$ is called \textbf{invariant} with respect to the flow $\varphi_t$ (a vector field $X$), if $\mu$ is an invariant measure of $\varphi_T$, for every $T\in\mathbb{R}$. Similarly, a measure $\mu$ is called \textbf{ergodic} with respect to the flow $\varphi_t$ (a vector field $X$), if $\mu$ is an ergodic invariant measure of $\varphi_T$, for every $T\in\mathbb{R}$. The set of all invariant measures and all ergodic measures of vector field are denoted by $\mathcal{M}(X)$, $\mathcal{E}(X)$ respectively. 
   
   Let $\mu$ be an invariant measure of flow $\varphi_t$. By the Oseledec Theorem \cite{OV}, for $\mu$-almost every $x\in M$, there exist a positive integer $k(x)\in[1,d]$, real numbers $\chi_1(x)<\chi_2(x)<\cdots<\chi_{k(x)}(x)$ and a measurable $\Phi_t$-invariant splitting $$T_xM=E_1(x)\oplus E_2(x)\oplus\cdots\oplus E_k(x)$$ such that $$\lim_{t\to\pm\infty}\dfrac{1}{t}\log\parallel\Phi_t(v) \parallel=\chi_i(x),~\forall~v\in E_i(x),~v\neq0,~i=1,2,\cdots,k(x).$$ The numbers $\lambda_1(x),\cdots,\lambda_{k(x)}(x)$ are called the {\bf Lyapunov exponents} at point $x$ of $\Phi_t$ with respect to $\mu$ and the vector formed by these numbers (counted with multiplicity, endowed with the increasing order) is called the \textbf{Lyapunov vector} at point $x$ of $\Phi_t$ with respect to $\mu$. The \textbf{index} of $\mu$, denoted by $\text{Ind}(\mu)$, is defined as $$\text{Ind}(\mu)\triangleq\sum_{\lambda_i(x)<0}\text{dim}E_i(x).$$ If $\mu$ is ergodic, then these numbers $k(x),\chi_1(x),\cdots,\chi_{k(x)}(x)$ are constants. In that case, we simply write $k,\chi_1,\cdots,\chi_k$. By the Poincar\'{e} Recurrence Theorem, for $\mu$-almost every $x\in M$, $$\lim_{t\to\pm\infty}\dfrac{1}{t}\log\parallel\Phi_t|_{\langle X(x)\rangle}\parallel=0,$$ where $\langle X(x)\rangle$ is the $1$-dimensional subspace of $T_xM$ spanned by the flow direction $X(x)$. Thus, it follows that there always exists a zero Lyapunov exponent for $\Phi_t$ along the flow direction. 
  
   For an ergodic invariant measure $\mu$ which is not supported on ${\rm Sing}(X)$, according to the definition of the linear Poincar\'{e} flow $\psi_t:\mathcal{N}\rightarrow\mathcal{N}$ and the Oseledec Theorem \cite{OV}, for $\mu$-almost every $x\in M\setminus{\rm Sing}(X)$, there exist a positive integer $k\in[1,d-1]$, real numbers $\chi_1<\chi_2<\cdots<\chi_k$ and a measurable $\psi_t$-invariant splitting (for simplicity, we omit the base point) $$\mathcal{N}=E_1\oplus E_2\oplus\cdots\oplus E_k$$ on normal bundle such that $$\lim_{t\to\pm\infty}\dfrac{1}{t}\log\parallel\psi_t(v) \parallel=\chi_i,~\forall~v\in E_i,~v\neq 0,~i=1,2,\cdots,k.$$ Accordingly, the numbers $\chi_1,\cdots,\chi_k$ are called the \textbf{Lyapunov exponents} of $\psi_t$ with respect to the ergodic measure $\mu$ and the vector formed by these numbers (counted with multiplicity, endowed with the increasing order) is called the \textbf{Lyapunov vector} of $\psi_t$ with respect to $\mu$. Recall that the scaled linear Poincar\'{e} flow $\psi_t^*:\mathcal{N}\rightarrow\mathcal{N}$ is defined as $$\psi^*_t(v)=\frac{\left\lVert X(x)\right\rVert}{\left\lVert X(\varphi_t(x))\right\rVert}\psi_t(v)=\frac{\psi_t(v)}{\left\lVert\Phi_t|_{\langle X(x)\rangle}\right\rVert},$$ where $\langle X(x)\rangle$ is the $1$-dimensional subspace of $T_xM$ spanned by the flow direction $X(x)$. Since the Lyapunov exponent of $\Phi_t$ along the flow direction is zero, i.e.,  $$\lim_{t\to\pm\infty}\dfrac{1}{t}\log\left\lVert\Phi_t|_{\langle X(x)\rangle}\right\rVert=0,$$ we obtain that $$\lim_{t\to\pm\infty}\dfrac{1}{t}\log\left\lVert\psi^*_t(v)\right\rVert=\lim_{t\to\pm\infty}\dfrac{1}{t}\left(\log\left\lVert\psi_t(v)\right\rVert-\log\left\lVert\Phi_t|_{\langle X(x)\rangle}\right\rVert\right)=\lim_{t\to\pm\infty}\dfrac{1}{t}\log\left\lVert\psi_t(v)\right\rVert.$$ Thus, for $\mu$-almost every $x\in M\setminus{\rm Sing}(X)$, there also exist a same positive integer $k\in[1,d-1]$, the same real numbers $\chi_1<\chi_2<\cdots<\chi_k$ and a measurable $\psi^*_t$-invariant splitting (for simplicity, we omit the base point) $$\mathcal{N}=E_1\oplus E_2\oplus\cdots\oplus E_k$$ on normal bundle such that $$\lim_{t\to\pm\infty}\dfrac{1}{t}\log\parallel\psi^*_t(v) \parallel=\chi_i,~\forall~v\in E_i,~v\neq 0,~i=1,2,\cdots,k.$$ The numbers $\chi_1,\cdots,\chi_k$ are the \textbf{Lyapunov exponents} of $\psi^*_t$ with respect to $\mu$ and the vector formed by these numbers (counted with multiplicity, endowed with the increasing order) is the \textbf{Lyapunov vector} of $\psi^*_t$ with respect to $\mu$. It means that the Lyapunov exponents of the scaled linear Poincar\'{e} flow and those of the linear Poincar\'{e} flow are the same. Hence, the scaled linear Poincar\'{e} flow and the linear Poincar\'{e} flow also have the same Oseledets splitting. Ignoring multiplicities, we also denote the Lyapunov exponents of (scaled) linear Poincar\'{e} flow with respect to $\mu$ by $\lambda_1\leq\lambda_2\leq\cdots\leq\lambda_{d-1}$, where $$\lambda_j=\chi_i,\text{ whenever $d_1+d_2+\cdots+d_{i-1}<j\leq d_1+d_2+\cdots+d_i$}.$$ 
   
   Given a periodic point $z$ of the vector field $X$, denote its period by $\pi(z)$. The Borel probability measure $$\mu_z=\dfrac{1}{\pi(z)}\int_{0}^{\pi(z)}\delta_{\varphi_t(z)}dt$$ is an invariant measure supported on the periodic orbit of $z$, where $\delta_y$ is the Dirac measure at $y$. By the Oseledec Theorem, the (scaled) linear Poincar\'{e} flow also has well-defined Lyapunov exponents and a Lyapunov vector with respect to the periodic invariant measure $\mu_z$.          
\begin{Definition}\label{Def:hyperbolicmeasure}
   An ergodic measure $\mu$ of the flow $\varphi_t$ is \emph{regular} if it is not supported on a singularity. A regular ergodic measure is \emph{hyperbolic}, if the Lyapunov exponents of the \emph{linear Poincar\'{e} flow} $\psi_t$ are all nonzero.
\end{Definition}
     
\begin{Remark}
   We can also define the hyperbolicity of an ergodic measure by using the tangent flow $\Phi_t={\rm d}\varphi_t$ as usual. However, for ergodic measures which are not supported on singularities, there will be a zero Lyapunov exponent of the tangent flow $\Phi_t$ along the flow direction. 
\end{Remark}

   Given an ergodic hyperbolic invariant measure $\mu$, denote by $\Gamma$ (called the Oseledec's basin of $\mu$ in \cite{WCZ21}) be the set of all points that are regular with respect to linear Poincar\'{e} $\psi_t$ in the sense of Oseledec \cite{OV}. For every $x\in\Gamma$, let $$\mathcal{N}_x=E_1(x)\oplus E_2(x)\oplus\cdots\oplus E_s(x)\oplus E_{s+1}(x)\oplus\cdots\oplus E_k(x)~~(k\in[0,d-1]\text{ is an integer})$$ be the Oseledec splitting corresponding to the distinct Lyapunov exponents $$\chi_1(\mu)<\chi_2(\mu)<\cdots<\chi_s(\mu)<0<\chi_{s+1}(\mu)<\cdots<\chi_k(\mu)$$ with multiplicities $d_1,d_2,\cdots,d_k\geq 1$. Hence $\text{dim}(E_i(x))=d_i$, for $i=1,2,\cdots,k$. Define the stable bundle and the unstable bundle respectively by $$E^s=E_1\oplus E_2\oplus\cdots\oplus E_s,~E^u=E_{s+1}\oplus E_{s+2}\oplus\cdots\oplus E_k.$$
\begin{Definition}\label{OSDS}
   Given an ergodic hyperbolic invariant measure $\mu$, the Oseledec splitting $$\mathcal{N}_\Gamma=E_1\oplus E_2\oplus\cdots\oplus E_s\oplus E_{s+1}\oplus\cdots\oplus E_k$$ of the linear Poincar\'{e} $\psi_t$ with respect to the ergodic hyperbolic invariant measure $\mu$ is called a \textbf{dominated splitting}, if there are constants $C\geq 1$ and $\lambda>0$ such that for every $x\in\Gamma$ and $t\geq 0$, the two subbundles $E_i$ and $E_{i+1}$ are dominated with respect to the linear Poincar\'{e} $\psi_t$ for each $i=1,\cdots,k-1$. 
\end{Definition}

   Now, we introduce some symbols and basic facts regarding the vector field $-X$. Given a vector field $X\in\mathfrak{X}^1(M)$, let $-X$ denote the vector field such that $-X(x)$ and $X(x)$ are opposite in directions and $\lvert-X(x)\rvert=\lvert X(x)\rvert$, for every $x\in M$. Consequently, the vector field $-X$ also generates a $C^1$ flow,  denoted by $\overline{\varphi}_t$. Correspondingly, we obtain the tangent flow $\overline{\Phi}_t$, the linear Poincar\'{e} flow $\overline{\psi}_t$, the extended linear Poincar\'{e} flow $\overline{\widetilde{\psi}}_t$ and the scaled linear Poincar\'{e} flow $\overline{\psi}^*_t$.

   According to the relationship between $\overline{\varphi}_t$ and $\varphi_t$, each ergodic invariant measure $\mu$ of $\varphi_t$ remains ergodic and invariant for $\overline{\varphi}_t$. Hence, by the Oseledec Theorem \cite{OV}, for $\mu$-almost every $x\in M$, there exist an integer $k\in[1,d]$, distinct real numbers $\overline{\chi}_1<\cdots<\overline{\chi}_{k}$ and a measurable $\overline{\Phi}_t$-invariant splitting $$T_xM=\overline{E}_1(x)\oplus\overline{E}_2(x)\oplus\cdots\oplus \overline{E}_k(x)$$ such that $$\lim_{t\to\pm\infty}\dfrac{1}{ t}\log\left\lVert\overline{\Phi}_t(v)\right\rVert=\overline{\chi}_i,~\forall~v\in\overline{E}_i(x),~v\neq 0,~i=1,2,\cdots,k.$$ These numbers $\overline{\chi}_1,\dots,\overline{\chi}_k$ are the \textbf{Lyapunov exponents} of $\overline{\Phi}_t$ at $x$ with respect to $\mu$. Their multiplicities are $$\overline{d}_i=\text{dim}(\overline{E}_i(x)),\text{ for $i=1,\cdots,k.$}$$ The \textbf{Lyapunov vector} is obtained by listing all exponents in increasing order, each repeated according to its multiplicity. By the Poincar\'{e} Recurrence Theorem, for $\mu$-almost every $x\in M$, the Lyapunov exponent of $\overline{\Phi}_t$ along the flow direction is zero. Namely,  $$\lim_{t\to\pm\infty}\dfrac{1}{t}\log\left\lVert\overline{\Phi}_t|_{\langle -X(x)\rangle}\right\rVert=0.$$ Taking multiplicities into account, we list all Lyapunov exponents of $\overline{\Phi}_t$ with respect to $\mu$ in non-decreasing order as $\overline{\lambda}_1 \leq\overline{\lambda}_2\leq\cdots\leq\overline{\lambda}_d$. That is, $$\overline{\lambda}_j=\overline{\chi}_i,\quad\text{for } \overline{d}_1+\overline{d}_2+\cdots+\overline{d}_{i-1}<j\leq \overline{d}_1+\overline{d}_2+\cdots+\overline{d}_i.$$ 

   For an ergodic invariant measure $\mu$ which is not supported on $\text{Sing}(X)$, the Oseledec Theorem \cite{OV} applied to the linear Poincar\'{e} flow $\overline{\psi}_t:\mathcal{N}\to\mathcal{N}$ implies that for $\mu$-almost every $x\in M\setminus\text{Sing}(X)$, there exist an integer $k\in[1,d-1]$, distinct real numbers $\overline{\chi}_1<\overline{\chi}_2<\cdots<\overline{\chi}_k$ and a measurable $\overline{\psi}_t$-invariant splitting (for simplicity, we omit the base point) $$\mathcal{N}=\overline{E}_1\oplus \overline{E}_2\oplus\cdots\oplus\overline{E}_k$$ on normal bundle such that $$\lim_{t\to\pm\infty}\dfrac{1}{t}\log\left\lVert\overline{\psi}_t(v)\right\rVert=\overline{\chi}_i,~\forall~v\in\overline{E}_i,~v\neq 0,~i=1,2,\cdots,k.$$ These numbers $\overline{\chi}_1<\overline{\chi}_2<\cdots<\overline{\chi}_k$ are the \textbf{Lyapunov exponents} of $\overline{\psi}_t$ with respect to $\mu$. 
      
   Recall the definition of the scaled linear Poincar\'{e} flow $$\overline{\psi}^*_t(v)=\frac{\left\lVert-X(x)\right\rVert}{\left\lVert-X(\overline{\varphi}_t(x))\right\rVert}\overline{\psi}_t(v)=\frac{\overline{\psi}_t(v)}{\left\lVert\overline{\Phi}_t|_{\langle-X(x)\rangle}\right\rVert}.$$ Since the Lyapunov exponent of $\overline{\Phi}_t$ along the flow direction vanishes,  $$\lim_{t\to\pm\infty}\dfrac{1}{t}\log\left\lVert\overline{\Phi}_t|_{\langle-X(x)\rangle}\right\rVert=0,$$ it follows that $$\lim_{t\to\pm\infty}\dfrac{1}{t}\log\left\lVert\overline{\psi}^*_t(v)\right\rVert=\lim_{t\to\pm\infty}\dfrac{1}{t}\left(\log\left\lVert\overline{\psi}_t(v)\right\rVert-\log\left\lVert\overline{\Phi}_t|_{\langle-X(x)\rangle}\right\rVert\right)=\lim_{t\to\pm\infty}\dfrac{1}{t}\log\left\lVert\overline{\psi}_t(v)\right\rVert.$$ Consequently, for $\mu$-almost every $x\in M\setminus\text{Sing}(X)$, the Oseledec splitting and the Lyapunov exponents for $\overline{\psi}^*_t$ are identical to those for $\overline{\psi}_t$. In particular, the numbers $\overline{\chi}_1,\cdots,\overline{\chi}_k$ are also the Lyapunov exponents of $\overline{\psi}^*_t$ with respect to $\mu$. Hence, the scaled linear Poincar\'{e} flow and the linear Poincar\'{e} flow share the same Lyapunov spectrum and Oseledec splitting.
   
   The identity $$\overline{\varphi}_{-t}=\varphi_t,\text{ for all }t \in\mathbb{R}$$ is a direct consequence of the relationship between the vector fields $-X$ and $X$. This implies for every vector $v\in T_xM$, $$\lim_{t\to\pm\infty}\dfrac{1}{ t}\log\left\lVert\overline{\Phi}_t(v)\right\rVert=\lim_{t\to\pm\infty}\dfrac{1}{t}\log\left\lVert\Phi_{-t}(v)\right\rVert=-\lim_{t\to\pm\infty}\dfrac{1}{-t}\log\left\lVert\Phi_{-t}(v)\right\rVert.$$ Consequently, for every vector $v\in\mathcal{N}$, $$\lim_{t\to\pm\infty}\dfrac{1}{ t}\log\left\lVert\overline{\psi}_t(v)\right\rVert=\lim_{t\to\pm\infty}\dfrac{1}{t}\log\left\lVert\psi_{-t}(v)\right\rVert=-\lim_{t\to\pm\infty}\dfrac{1}{-t}\log\left\lVert\psi_{-t}(v)\right\rVert$$ and $$\lim_{t\to\pm\infty}\dfrac{1}{ t}\log\left\lVert\overline{\psi}^*_t(v)\right\rVert=\lim_{t\to\pm\infty}\dfrac{1}{t}\log\left\lVert\psi^*_{-t}(v)\right\rVert=-\lim_{t\to\pm\infty}\dfrac{1}{-t}\log\left\lVert\psi^*_{-t}(v)\right\rVert.$$ Thus, if $\chi_1<\cdots<\chi_k$ are the Lyapunov exponents of $\psi_t$ (or $\psi^*_t$) with respect to $\mu$, then the Lyapunov exponents $\overline{\chi}_i$ of $\overline{\psi}_t$ (or $\overline{\psi}^*_t$) satisfy the reversed-order relation $$\overline{\chi}_i=-\chi_{k-i+1},\text{ for every }i=1,2,\cdots,k.$$ Moreover, the corresponding Oseledec subspaces and their dimensions are reversed accordingly:   $$\overline{E}_i=E_{k-i+1},~\overline{d}_i=\text{dim}\overline{E}_i=d_{k-i+1},~\text{ for }i=1,2,\cdots,k.$$ When the Lyapunov exponents are listed with multiplicity in non-decreasing order as $$\lambda_1\le\cdots\le\lambda_{d-1}\text{ and }\overline{\lambda}_1\le\cdots\le\overline{\lambda}_{d-1},$$ we obtain $$\overline{\lambda}_i=-\lambda_{d-i},\text{ for }i=1,2,\cdots,d-1.$$ Hence, every ergodic hyperbolic invariant measure of $\varphi_t$ is also an ergodic hyperbolic invariant measure of $\overline{\varphi}_t$. 
  
\section{Lyapunov Norms and Shadowing Lemma for Vector Fields}\label{LP}

\subsection{Lyapunov Norms and Lyapunov Charts}
   In this section, we only consider $C^1$ vector field $X\in\mathfrak{X}^1(M)$. Since the linear Poincar\'{e} flow (see Definition \ref{Def:linearPoincare}) is not defined at singularities, some uniformity is lost. However, by considering the scaled linear Poincar\'{e} flow, one can recover certain uniform properties. 
   
   On the tangent bundle $TM$, the \emph{Sasaki metric} (see \cite[ Subsection 2.1]{BMW12}), denoted by $d_{S}$, is induced by the Riemannian metric on $M$ and proves very useful. By the compactness of $M$, there exists a small constant $\rho_S$ such that for any two points $x,~y\in M$ with $d(x,y)<\rho_S$, there exists a unique geodesic joining $y$ to $x$. Moreover, if $u_x\in T_xM$ and $u_y\in T_yM$ satisfy $d_{S}(u_x,u_y)\leq\rho_S$ (equivalently, if $|u_x-u'_y|\leq\rho_S$, where $u'_y\in T_xM$ is the parallel transport of $u_y$ along this geodesic from $y$ to $x$), then $$\dfrac{\left(d(x,y)+\lvert u_x-u'_y\rvert\right)}{2}\leq d_{S}(u_x,u_y)\leq 2\left(d(x,y)+\lvert u_x-u'_y\rvert\right).$$ Furthermore, there exists a constant $K_{G^1}$ such that for any two unit vectors $u_1,~u_2\in TM$, $$d_{G^1}(\langle u_1\rangle,\langle u_2\rangle)\leq K_{G^1}\cdot d_{S}(u_1,u_2),$$ where $d_{G^1}$ is the Grassman distance on the $1$-dimensional Grassmann bundle over $M$, i.e., $$G^1=\left\{N\subset T_xM: N \text{ is a $1$-dimensional linear subspace of $T_xM$}, x\in M\right\}.$$ Since $M$ is compact, there exists a constant $K_0$ such that $$\max_{x\in M}\left\{\lvert X(x)\rvert,~\lVert DX(x)\rVert\right\}\leq K_0.$$ The following lemma establishes the relationship between the Sasaki metric and the Riemannian metric.    
\begin{Lemma}\label{Dlgx}{\rm\cite[Lemma 3.1]{LLL24}}
   There exist constants $K_1>K_0$ and $\beta^*_0>0$ such that for every $x\in M\setminus\text{Sing}(X)$ and every $y\in B\left(x,\beta^*_0\left\lvert X(x)\right\rvert\right)$, the following estimates hold:
\begin{itemize}
   \item[{\rm (1)}] $1-K_1\cdot \dfrac{d(x,y)}{\left\lvert X(x)\right\rvert}\leq\dfrac{\left\lvert X(x)\right\rvert}{\left\lvert X(y)\right\rvert}\leq1+K_1\cdot\dfrac{d(x,y)}{\left\lvert X(x)\right\rvert}${\rm;}
	
   \item [{\rm (2)}] $d_S\left(\dfrac{X(x)}{\left\lvert X(x)\right\rvert},\dfrac{X(y)}{\left\lvert X(y)\right\rvert}\right)\leq K_1\cdot\dfrac{d(x,y)}{\left\lvert X(x)\right\rvert}${\rm;}
	
   \item [{\rm (3)}] $d_{G^1}\left(\langle X(x)\rangle,\langle X(y)\rangle\right)\leq K_1\cdot\dfrac{d(x,y)}{\left\lvert X(x)\right\rvert}${\rm.}
\end{itemize}   	
\end{Lemma}

   Let $r_0$ be the radius such that the exponential map $$\exp:~TM\mapsto M$$ is a $C^{\infty}$ diffeomorphism on $T_xM(r_0)$. By reducing $\beta^*_0>0$ if necessary, we may assume that $$10\beta^*_0K_1<\min\{r_0,1\}.$$ Thus, Item $(1)$ of Lemma \ref{Dlgx} implies that there are no singularities in $B(x,\beta^*_0\lvert X(x)\rvert)$. To regain uniformity, we next analyze the sectional Poincar\'{e} flow using local charts defined on scaled neighborhoods. These charts, known as Liao scaled charts, were introduced in \cite{LST79A,LST85,Liao89}. Similar discussions can be found in \cite{CY2017,GY,WW19}.
   
   Let $\left\{e_1,\cdots,e_d\right\}$ be an orthonormal basis of $\mathbb{R}^d$. For every $x\in M$, by taking an orthonormal basis $\left\{e^x_1,\cdots,e^x_d\right\}$ of $T_xM$, we can define a linear isometry $C_x:\mathbb{R}^d\to T_xM$, which serves as a change of coordinates from the Euclidean space to the tangent space, such that $$C_x(e_i)=e^x_i,\text{ for $i=1,2,\cdots,d$}.$$ Then, $$\langle C_x(u), C_x(v)\rangle_x=\langle u,v\rangle,~\forall~u,~v\in\mathbb{R}^d,$$ where $\langle\cdot,\cdot\rangle$ is the standard scalar product on $\mathbb{R}^d$. Define $${\rm Exp}_x=\exp_x\circ C_x:~\mathbb{R}^d\to M$$ and denote $$\mathbb{R}^d(r)=\left\{v\in\mathbb{R}^d:~\lvert v\rvert\leq r\right\}.$$ Then, for every $x\in M$, ${\rm Exp}_x|_{\mathbb{R}^d(r_0)}$ is a $C^{\infty}$ diffeomorphism from $\mathbb{R}^d(r_0)$ to $B(x,r_0)$. This allows us to lift the flow locally from the neighborhood $B(x,r_0)$ to $\mathbb{R}^d(r_0)$ via ${\rm Exp}_x$. That is, if there exist $y\in\mathbb{R}^d$ and $t_1<0<t_2$ such that $$\varphi_t({\rm Exp}_x(y))\in B(x,r_0),~\text{ for each $t\in[t_1,t_2]$},$$ then we define $$\widetilde{\varphi_{x,t}}(y)=\left({\rm Exp}_x|_{\mathbb{R}^d(r_0)}\right)^{-1}\circ\varphi_t\circ{\rm Exp}_x(y),~\text{ for each $t\in[t_1,t_2]$}.$$ In this Liao scaled chart at $x$, the flow $\varphi_t$ generated by the vector field $X\in\mathfrak{X}^1(M)$ satisfies the differential equation $$\dfrac{dz}{dt}=\widehat{X}_x(z),$$ where $\widehat{X}_x(z)=D\left({\rm Exp}_x|_{\mathbb{R}^d(r_0)}\right)^{-1}\circ X\circ{\rm Exp}_x(z)$ and ${\rm Exp}_x(0)=x$, $\left\lvert\widehat{X}_x(0)\right\rvert=\left\lvert X(x)\right\rvert$. By the compactness of $M$, there exists a constant $K_{{\rm E}}>1$ such that $$\max_{x\in M}\left\{\left\lVert D{\rm Exp}_x|_{\mathbb{R}^d(r_0)}\right\rVert,~\left\lVert D\left({\rm Exp}_x|_{\mathbb{R}^d(r_0)}\right)^{-1}\right\rVert\right\}<K_{{\rm E}}.$$ Hence, increasing $K_0$ if necessary, we may further assume that $$\max_{x\in M}\left\{\max_{p\in\mathbb{R}^d(r_0)}\left\{\left\lvert\widehat{X}_x(p)\right\rvert,~\left\lVert D\widehat{X}_x(p)\right\rVert\right\}\right\}<K_0.$$ Therefore, $$\left\lvert\widehat{X}_x(p)-\widehat{X}_x(0)\right\rvert<K_0\lvert p\rvert\leq K_0r_0.$$
   
   For every regular point $x\in M$, take $$e^x_1=\dfrac{X(x)}{\left\lvert X(x)\right\rvert}.$$ Then, $\widehat{X}_x(0)=\left(\lvert X(x)\rvert,0,0,\cdots,0\right)$. Given $x\in M\setminus\text{Sing}(X)$, for every $y\in B(x,\beta^*_0\left\lvert X(x)\right\rvert)$, we choose an orthonormal basis $\left\{e^y_1,\cdots,e^y_d\right\}$ of $T_yM$ such that $e^y_1=\dfrac{X(y)}{\lvert X(y)\rvert}$ and $$d_S\left(e^{y_1}_j,e^{y_2}_j\right)<2K_1\dfrac{d(y_1,y_2)}{\lvert X(x)\rvert},\text{ for any two points $y_1,~y_2\in B(x,\beta^*_0\lvert X(x)\rvert)$, }j=1,2,\cdots,d.$$ Using these orthonormal bases on $T_{B(x,\beta^*_0|X(x)|)}M$, we obtain a uniform $C^1$ estimate for the diffeomorphism ${\rm Exp}_x$, as stated in the following Lemma \ref{Eugx}.    
\begin{Lemma}\label{Eugx}{\rm\cite[Lemma 3.2]{LLL24}}
   There exists $K_2>1$ such that for every $x\in M\setminus\text{Sing}(X)$ and any two points $y_1,~y_2\in B(x,\beta^*_0\lvert X(x)\rvert)$, the following estimate holds: $$\left\lVert\left({\rm Exp}^{-1}_{y_1}\circ{\rm Exp}_{y_2}-id_{\mathbb{R}^d}\right)|_{\mathbb{R}^d(r_0/2)}\right\rVert_{C^1}\leq K_2\dfrac{d(y_1,y_2)}{\lvert X(x)\rvert},$$ where $\lVert\cdot\rVert_{C^1}$ is the $C^1$ norm.	
\end{Lemma}

\subsection{Shadowing lemma}
   The concept of \emph{shadowing} was first introduced by Sina\u{\i} \cite{Si72}, who proved that Anosov diffeomorphisms have the shadowing property and, moreover, that every pseudo-orbit is shadowed by a unique true orbit. Shadowing describes a situation in which a true orbit stays uniformly close to a given pseudo-orbit. The shadowing lemma for uniformly hyperbolic diffeomorphisms was fully developed by Bowen \cite{B70}; see also, e.g., \cite{CLP89,P99} for related references. For flows, however, the shadowing lemma is considerably more complicated. 
   
   In \cite{LST80}, Liao proposed the idea of \textquotedblleft quasi-hyperbolic strings\textquotedblright, showing that a special type of pseudo-orbit can still be shadowed by a true orbit in non-uniform hyperbolic dynamical systems. The first shadowing lemma for a single quasi-hyperbolic string was given by Liao for both diffeomorphisms \cite{LST79} and flows \cite{LST80,LST85}. It asserts that if the head and tail of a single quasi-hyperbolic string are sufficiently close, then the string is shadowed by a periodic orbit. For diffeomorphisms, a more general shadowing lemma for a sequence of quasi-hyperbolic strings was later given by Gan \cite{Gan2002}. Han and Wen \cite{HW18} extended Gan's result from diffeomorphisms to $C^1$ vector fields, thereby proving a shadowing lemma for sequences of quasi-hyperbolic strings for flows. In what follows, we review the basic form of the shadowing lemma for flows as established by Liao \cite{LST80,LST85}, which is fundamental to our paper.    
\begin{Definition}
   Let $\Lambda$ be an invariant set and $E\subset\mathcal{N}_{\Lambda\setminus\text{Sing}(X)}$ be an invariant subbundle of the scaled linear Poincar\'{e} flow $\psi^*_t$. For $C>0$, $\eta>0$, and $T>0$, a point $x\in\Lambda\setminus\text{Sing}(X)$ is called $(C,\eta,T,E)$-$\psi^*_t$-\textbf{contracting} if there exists an increasing sequence of times $$0=t_0<t_1<\cdots<t_n<\cdots, \text{ with $t_n \to +\infty$ as $n \to +\infty$},$$ such that $t_{i+1}-t_i\leq T$, for every $i\in\mathbb{N}$, and $$\prod_{i=0}^{n-1}\left\lVert\psi^*_{t_{i+1}-t_i}|_{E\left(\varphi_{t_i}(x)\right)}\right\rVert\leq Ce^{-\eta t_n},~\text{for every }n\in\mathbb{N}.$$ A point $x\in\Lambda\setminus\text{Sing}(X)$ is called $(C,\eta,T,E)$-$\psi^*_t$-\textbf{expanding} if it is $(C,\eta,T,E)$-$\psi^*_t$-contracting for the vector field $-X$. 	
\end{Definition}

\begin{Definition}
   Given $\eta>0$ and $T>0$. For every $x\in M\setminus\text{Sing}(X)$ and $T_0>T$, the orbit arc $\varphi_{[0,T_0]}$ is called $(\eta,T)$-$\psi^*_t$-\textbf{quasi-hyperbolic} with respect to a direct sum splitting $\mathcal{N}_x=E(x)\oplus F(x)$, if there exists a time partition $$0=t_0<t_1<\cdots<t_k=T_0,\text{ with $t_{i+1}-t_i\leq T$ for $i=0,1,\cdots,k-1$}$$ such that for every $n=0,1,\cdots,k-1$, the following inequalities hold $$\prod_{i=0}^{n-1}\left\lVert\psi^*_{t_{i+1}-t_i}|_{E\left(\varphi_{t_i}(x)\right)}\right\rVert\leq e^{-\eta t_n},\quad\prod_{i=n}^{k-1}m\left(\psi^*_{t_{i+1}-t_i}|_{F\left(\varphi_{t_i}(x)\right)}\right)\geq e^{\eta(t_k-t_n)},$$ $$\dfrac{\left\lVert\psi^*_{t_{n+1}-t_n}|_{E\left(\varphi_{t_n}(x)\right)}\right\rVert}{m\left(\psi^*_{t_{n+1}-t_n}|_{F\left(\varphi_{t_n}(x)\right)}\right)}\leq e^{-\eta(t_{n+1}-t_n)}.$$      	
\end{Definition}

   The following shadowing lemma for singular flows was established by Liao \cite{LST85}, which provides a method for finding periodic points. 
\begin{Theorem}\label{Shadow}
   Assume $X\in\mathfrak{X}^1(M)$, and let $\Lambda\subset M\setminus\text{Sing}(X)$ be an invariant set admitting a dominated splitting $\mathcal{N}_\Lambda=E\oplus F$ with respect to the scaled linear Poincar\'{e} flow. Given $\eta>0$ and $T>0$, for every $\alpha>0$ and $\varepsilon>0$, there exists $\mathcal{D}=\mathcal{D}(\alpha,\varepsilon)>0$ such that for every $(\eta,T)$-$\psi^*_t$-quasi-hyperbolic orbit segment $\varphi_{[0,T_0]}(x)$ satisfying that 
\begin{itemize}
   \item $d(x,\text{Sing}(X))>\alpha$ and $d(\varphi_{T_0}(x),\text{Sing}(X))>\alpha${\rm;}
		
   \item $x\in\Lambda$, $\varphi_{T_0}(x)\in\Lambda$ and $d(x,\varphi_{T_0}(x))<\mathcal{D}${\rm;} 
\end{itemize}  
   there exist a strictly increasing $C^1$ function $\theta:[0,T_0]\to\mathbb{R}$ and a periodic point $p\in M\setminus\text{Sing}(X)$ such that   
\begin{description}
   \item[(1)] $\theta(0)=0$ and $1-\varepsilon<\theta'(t)<1+\varepsilon$, for every $t\in[0,T_0]${\rm;}
		
   \item[(2)] $p$ is a periodic point with period $\theta(T_0)$ {\rm:} $\varphi_{\theta(T_0)}(p)=p${\rm;}
		
   \item[(3)] $d(\varphi_t(x),\varphi_{\theta(t)}(p))<\varepsilon\left\lvert X(\varphi_t(x))\right\rvert$, for every $t\in[0,T_0]${\rm.}
\end{description}    
\end{Theorem}

\section{Density of periodic measures: proof of Theorem \ref{ThmA}}\label{MA}
   This section is devoted to proving Theorem \ref{ThmA}. Before proceeding, we first discuss the Oseledec splitting in the case where it is a dominated splitting. Let $\varphi_t$ be the $C^1$ flow generated by a vector field $X\in\mathfrak{X}^1(M)$, and let $\mu$ be an ergodic hyperbolic invariant measure that is not supported on any singularity. Assume that the Lyapunov exponents of the scaled linear Poincar\'{e} flow $\psi^*_t$ with respect to $\mu$ are $\chi_1<\chi_2<\cdots<\chi_k$. Then, on the $\Gamma$ of Oseledec regular points for $\mu$, there exists a measurable $\psi^*_t$-invariant splitting $$\mathcal{N}_\Gamma=E_1\oplus E_2\oplus\cdots\oplus E_k$$ of normal bundle, which satisfies $$\lim_{t\to\pm\infty}\dfrac{1}{t}\log\left\lVert\psi^*_t(v)\right \rVert=\chi_i,~\forall~v\in E_i,~v\neq 0,~i=1,2,\cdots,k.$$       
\begin{Lemma}\label{time}
   If the Oseledec splitting $$\mathcal{N}_\Gamma=E_1\oplus E_2\oplus\cdots\oplus E_k$$ of an ergodic hyperbolic invariant regular measure $\mu$ is a dominated splitting, then for every $0<\epsilon\ll\min\limits_{1\leq j\leq k}\{\lvert\chi_j\rvert\}$ and every $\delta\in(0,1)$, there exists a positive number $L=L(\epsilon,\delta)$ such that
\begin{itemize}
   \item for every $T\geq L$, there exists a measurable set $\Gamma^T=\Gamma^T(\epsilon,\delta)\subset\Gamma$ with $\mu\left(\Gamma^T\right)\geq 1-\delta${\rm ;}
		
   \item there exists a natural number $N=N(T)$ such that for every integer $n\geq N$ and every $x\in\Gamma^T$, $$e^{(\chi_j-\epsilon)nT}\leq\prod_{i=0}^{n-1}m\left(\psi^*_T|_{E_j\left(\varphi_{iT}(x)\right)}\right)\leq\prod_{i=0}^{n-1}\left\lVert\psi^*_T|_{E_j\left(\varphi_{iT}(x)\right)}\right\rVert\leq e^{(\chi_j+\epsilon)nT},~j=1,2,\cdots,k.$$
\end{itemize}     	
\end{Lemma}      

\begin{proof}
   By the Oseledec Theorem, for the scaled linear Poincar\'{e} flow $\psi^*_t$, for every $x\in\Gamma$,
\begin{equation*}\label{*}
   \lim_{t\to\pm\infty}\dfrac{1}{t}\log m\left(\psi^*_t|_{E_j(x)}\right)=\lim_{t\to\pm\infty}\dfrac{1}{t}\log\left\lVert\psi^*_t|_{E_j(x)}\right\rVert=\chi_j,\text{ for }j=1,2,\cdots,k.\tag{$*$} 
\end{equation*}   
   Now, for every $0<\epsilon\ll\min\limits_{1\leq j\leq k}\{\lvert\chi_j\rvert\}$ and every $\delta\in(0,1)$, we proceed to prove the lemma. By equation $(*)$, for each $x\in\Gamma$ and each $j=1,2,\cdots,k$, $$\lim_{t\to\pm\infty}\left\lvert\dfrac{1}{t}\log\left\lVert\psi^*_t|_{E_j(x)}\right\rVert-\chi_j\right\rvert=0,~\lim_{t\to\pm\infty}\left\lvert\dfrac{1}{t}\log m\left(\psi^*_t|_{E_j(x)}\right)-\chi_j\right\rvert=0.$$ Consequently, $$\lim_{t\to\pm\infty}\int_{M}\left\lvert\dfrac{1}{t}\log\left\lVert\psi^*_t|_{E_j(x)}\right\rVert-\chi_j\right\rvert d\mu(x)=0,~\lim_{t\to\pm\infty}\int_{M}\left\lvert\dfrac{1}{t}\log m\left(\psi^*_t|_{E_j(x)}\right)-\chi_j\right\rvert d\mu(x)=0.$$ Thus, there exist positive numbers $L^+_j$ and $L^-_j$ such that the inequalities  $$\int_{M}\left\lvert\dfrac{1}{T}\log\left\lVert\psi^*_T|_{E_j(x)}\right\rVert-\chi_j\right\rvert d\mu(x)\leq\dfrac{\delta\epsilon}{12(1+k)},~\int_{M}\left\lvert\dfrac{1}{T}\log m\left(\psi^*_T|_{E_j(x)}\right)-\chi_j\right\rvert d\mu(x)\leq\dfrac{\delta\epsilon}{12(1+k)}$$ hold for all $T\geq L^+_j$ and $T\geq L^-_j$, respectively. Since the Oseledec splitting $$\mathcal{N}_\Gamma=E_1\oplus E_2\oplus\cdots\oplus E_k$$ is dominated, by \cite{BDV05}, it is continuous. Hence, for each fixed $j\in\{1,2,\cdots,k\}$ and every integer $T\geq\max\{L^+_j,~L^-_j\}$, the two function $$f_{j,T}(x):=\left\lvert\dfrac{1}{T}\log\left\lVert\psi^*_T|_{E_j(x)}\right\rVert-\chi_j\right\rvert,~\bar{f}_{j,T}(x):=\left\lvert\dfrac{1}{T}\log m\left(\psi^*_T|_{E_j(x)}\right)-\chi_j\right\rvert$$ are continuous on $\Gamma$. 
    
   On the other hand, the measure $\mu$ is invariant under $\varphi_T$. By the Birkhoff Ergodic Theorem, there exist two measurable functions $g_{j,T}$ and $\bar{g}_{j,T}$ such that for $\mu$-almost every $x\in M$, $$\lim_{n\to+\infty}\dfrac{1}{n}\sum_{i=0}^{n-1}f_{j,T}(\varphi_{iT})=\lim_{n\to+\infty}\dfrac{1}{nT}\sum_{i=0}^{n-1}\left\lvert\log\left\lVert\psi^*_T|_{E_j\left(\varphi_{iT}(x)\right)}\right\rVert-T\chi_j\right\rvert=g_{j,T}(x),$$ $$\lim_{n\to+\infty}\dfrac{1}{n}\sum_{i=0}^{n-1}\bar{f}_{j,T}(\varphi_{iT})=\lim_{n\to+\infty}\dfrac{1}{nT}\sum_{i=0}^{n-1}\left\lvert\log m\left(\psi^*_T|_{E_j\left(\varphi_{il}(x)\right)}\right)-T\chi_j\right\rvert=\bar{g}_{j,T}(x).$$ Moreover, it follows that
\begin{equation*}
   \int_Mf_{j,T}(x)d\mu(x)=\int_M\left\lvert\dfrac{1}{T}\log\left\lVert\psi^*_T\lvert_{E_j(x)}\right\rVert-\chi_j\right\rvert d\mu(x)=\int_Mg_{j,T}(x)d\mu(x)\leq\dfrac{\delta\epsilon}{36(1+k)}.\tag{$**$}
\end{equation*} 
\begin{equation*}
   \int_M\bar{f}_{j,T}(x)d\mu(x)=\int_M\left\lvert\dfrac{1}{T}\log m\left(\psi^*_T|_{E_j(x)}\right)-\chi_j\right\rvert d\mu(x)=\int_M\bar{g}_{j,T}(x)d\mu(x)\leq\dfrac{\delta\epsilon}{36(1+k)}.\tag{\dag}
\end{equation*}
   Define $$\Gamma^+_{j,T}=\left\{x\in\Gamma:~g_{j,T}(x)\leq\dfrac{\epsilon}{2}\right\},\qquad\Gamma^-_{j,T}=\left\{x\in\Gamma:~\bar{g}_{j,T}(x)\leq\dfrac{\epsilon}{2}\right\}.$$ By inequality $(**)$ and $(\dag)$, we have $$\mu\left(\Gamma^+_{j,T}\bigcap\Gamma^-_{j,T}\right)\geq 1-\dfrac{\delta}{9(1+k)}$$ and for every $x\in\Gamma^+_{j,T}\bigcap\Gamma^-_{j,T}$, $$\lim_{n\to+\infty}\dfrac{1}{nT}\sum_{i=0}^{n-1}\left\lvert\log\left\lVert\psi^*_T|_{E_j\left(\varphi_{iT}(x)\right)}\right\rVert-T\chi_j\right\rvert\leq\dfrac{\epsilon}{2},~\lim_{n\to+\infty}\dfrac{1}{nT}\sum_{i=0}^{n-1}\left\lvert\log m\left(\psi^*_T|_{E_j\left(\varphi_{iT}(x)\right)}\right)-T\chi_j\right\rvert\leq\dfrac{\epsilon}{2}.$$ In other words, for every $x\in\Gamma^+_{j,T}\bigcap\Gamma^-_{j,T}$, there exists $N_j(x)$ such that for every integer $n\geq N_j(x)$, the following inequalities hold $$\prod_{i=0}^{n-1}\left\lVert\psi^*_T|_{E_j\left(\varphi_{iT}(x)\right)}\right\rVert\leq e^{(\chi_j+\epsilon)nT},~\prod_{i=0}^{n-1}m\left(\psi^*_T|_{E_j\left(\varphi_{iT}(x)\right)}\right)\geq e^{(\chi_j-\epsilon)nT}.$$ 
   
   Let $$L:=\max_{1\leq j\leq k}\{L^+_j,~L^-_j\}.$$ For every $T\geq L$, define $$\Lambda^T:=\bigcap_{j=1}^k\left(\Gamma^+_{j,T}\bigcap\Gamma^-_{j,T}\right).$$ Then, $\Lambda^T\subset\Gamma$ and satisfies $$\mu\left(\Lambda^T\right)>1-\dfrac{\delta}{9}.$$ Now, for each $x\in\Lambda^T$, set $$N(x)=\max_{1\leq j\leq k}\left\{N_j(x)\right\}.$$ Then for every $j=1,2,\cdots,k$ and every integer $n\geq N(x)$,
\begin{equation*}
   e^{(\chi_j-\epsilon)nT}\leq\prod_{i=0}^{n-1}m\left(\psi^*_T|_{E_j\left(\varphi_{iT}(x)\right)}\right)\leq\prod_{i=0}^{n-1}\left\lVert\psi^*_T|_{E_j\left(\varphi_{iT}(x)\right)}\right\rVert\leq e^{(\chi_j+\epsilon)nT}.\tag{\dag\dag}
\end{equation*}   
   For $x\in\Lambda^T$, define $$\overline{N}(x):=\min\left\{N(x):~\text{ the inequality $(\dag\dag)$ holds for all integer } n\geq N(x)\right\}.$$ Denote $$\Lambda^T_n=\{x\in\Lambda^T:~\overline{N}(x)\leq n\}.$$ Then, $\Lambda^T_n\subset\Gamma^T_{n+1}$ (or more precisely, $\Lambda^T_n$ is increasing in $n$), and $$\Lambda^T=\bigcup_{n=1}^{+\infty}\Lambda^T_n.$$ Consequently, there exists an integer $n_0$ such that $\mu\left(\Lambda^T_{n_0}\right)>1-\delta$. To complete the proof of the lemma, we set $$N=N(T)=n_0,\quad\Gamma^T=\Lambda^T_{n_0}.$$  
\end{proof}   

   Now, we complete the proof of Theorem \ref{ThmA}.   
\begin{proof}[Proof of Theorem \ref{ThmA}]
   Let $C(M)$ denote the space of all continuous real-valued functions on $M$, and let $\{f_i\}_{i=1}^{+\infty}$ be a countable dense subset of $C(M)$. For any two measures $\mu,~\nu\in\mathcal{M}(X)$, define $$d_{\mathcal{M}}(\mu,\nu)=\sum_{i=1}^{+\infty}\dfrac{\left\lvert\int_{M}f_id\mu-\int_{M}f_id\nu\right\rvert}{2^i\lVert f_i\rVert}.$$ Then $d_{\mathcal{M}}(\cdot,\cdot)$ is a metric that induces the weak$^*$ topology on $\mathcal{M}(X)$. Given an ergodic hyperbolic invariant regular measure $\mu$, for every $\varepsilon>0$, we shall prove that there exists a periodic measure $\mu_p$ satisfying $$d_{\mathcal{M}}(\mu,\mu_p)<\varepsilon.$$ 
     
   First, choose $n$ large enough such that for every invariant measure $\nu$, the following holds
\begin{equation*}
   \sum_{i=n+1}^{+\infty}\dfrac{\left\lvert\int_{M}f_id\mu-\int_{M}f_id\nu\right\rvert}{2^i\left\lVert f_i\right\rVert}\leq\sum_{i=n+1}^{+\infty}\dfrac{1}{2^{i-1}}<\dfrac{\varepsilon}{2}.\tag{$\vartriangle$}
\end{equation*} 
   According to the Birkhoff Ergodic Theorem, there exists a $\varphi$-invariant set $\Lambda_B$ with full $\mu$-measure such that for every $x\in\Lambda_B$ and every $f\in C(M)$, $$\lim_{T\to+\infty}\dfrac{1}{T}\int^{T}_{0}f(\varphi_t(x))dt=\int_{M}fd\mu.$$ Thus, there exists $T_1>1$ such that for every $T>T_1$, every $x\in\Lambda_B$ and every $i=1,2,\cdots,n$,
\begin{equation*}
   \left\lvert\dfrac{1}{T}\int^{T}_{0}f_i(\varphi_t(x))dt-\int_Mf_id\mu\right\rvert<\dfrac{\varepsilon}{4n}\cdot\min_{1\leq i\leq n}\left\{2^i\left\lVert f_i\right\rVert\right\}. \tag{$\vartriangle\vartriangle$}
\end{equation*}   
	  
   Assuming that the Oseledec splitting $$\mathcal{N}_\Gamma=E_1\oplus E_2\oplus\cdots\oplus E_k$$ of the scaled linear Poincar\'{e} flow with respect to the ergodic hyperbolic invariant regular measure $\mu$ is a dominated splitting, we regroup it as $$\mathcal{N}_\Gamma=E_1\oplus E_2\oplus\cdots\oplus E_s\oplus E_{s+1}\oplus\cdots\oplus E_k$$ and write $$\mathcal{N}_\Gamma=E^s\oplus E^u,$$ where $E^s=E_1\oplus E_2\oplus\cdots\oplus E_s\text{ and } E^u=E_{s+1}\oplus E_{s+2}\oplus\cdots\oplus E_k$. Then, $$\mathcal{N}_\Gamma=E^s\oplus E^u$$ is a dominated splitting of the scaled linear Poincar\'{e} flow. Let $$\chi=\min_{1\leq j\leq k}\{\lvert\chi_j\rvert\}.$$ By Lemma \ref{time}, for every $0<\epsilon\ll\chi$ and every $\delta\in(0,1)$, there exists a positive number $L=L(\epsilon/2,\delta)$ such that
\begin{itemize}
   \item for every $T_0\geq L$, there exists a measurable set $\Gamma^{T_0}=\Gamma^{T_0}(\epsilon/2,\delta)\subset\Gamma$ with $\mu\left(\Gamma^{T_0}\right)\geq 1-\delta${\rm;}
		
   \item there exists a positive integer $N=N(T_0)$ such that for every integer $J\geq N$ and every $x\in\Gamma^{T_0}$, $$\prod_{i=0}^{J-1}\left\lVert\psi^*_{T_0}|_{E^s\left(\varphi_{iT_0}(x)\right)}\right\rVert\leq e^{-(\chi-\epsilon/2)JT_0},~\prod_{i=0}^{J-1}m\left(\psi^*_{T_0}|_{E^u\left(\varphi_{iT_0}(x)\right)}\right)\geq e^{(\chi-\epsilon/2)JT_0},$$ $$\text{ and }\dfrac{\left\lVert\psi^*_{T_0}|_{E^s(x)}\right\rVert}{m\left(\psi^*_{T_0}|_{E^u(x)}\right)}\leq e^{-T_0(\chi-\epsilon/2)}.$$     	  
\end{itemize}
   Fix an integer $T_0\geq\max\left\{T_1,~L(\epsilon/2,\delta)\right\}$ and set $\eta_0=(\chi-\epsilon/2)T_0$. For every $C>0$, we define the {\bf Pesin block } $\Lambda^{T_0}_{\eta_0}(C)$ by
\begin{equation*}
\begin{aligned}
   \Lambda^{T_0}_{\eta_0}(C)=\Bigg\{x\in\Gamma:&\prod_{i=0}^{J-1}\left\lVert\psi^*_{T_0}|_{E^s\left(\varphi_{iT_0}(x)\right)}\right\rVert\leq Ce^{-J\eta_0},~\forall~ J\geq 1,\\ &\prod_{i=0}^{J-1}m\left(\psi^*_{T_0}|_{E^u\left(\varphi_{iT_0}(x)\right)}\right)\geq C^{-1}e^{J\eta_0},~\forall~ J\geq 1,~d(x,\text{Sing}(X))\geq\dfrac{1}{C}\Bigg\}.
\end{aligned}   
\end{equation*}
   According to \cite[Proposition 5.3]{WYZ}, the set $\Lambda^{T_0}_{\eta_0}(C)$ is compact and $$\mu\left(\Lambda^{T_0}_{\eta_0}(C)\right)\to\mu\left(\Gamma^{T_0}\right)\text{~~as~~}C\to+\infty.$$ 
   
   Denote $$K=\max_{1\leq i\leq n}\left\{\lVert f_i\rVert\right\}.$$ Take $\gamma>0$ sufficiently small such that $$(2 K+1)\gamma<\dfrac{\varepsilon}{4n}\cdot\min_{1\leq i\leq n}\left\{2^i\left\lVert f_i\right\rVert\right\}.$$ Recall that $\lvert X(x)\rvert\leq K_0$, for every $x\in M$. Since each $f_i$ ($1\leq i\leq n$) is uniformly continuous on $M$, there exists $\xi\in[0,\gamma]$ sufficiently small such that for any $x,y\in M$ with $d(x,y)\leq\xi K_0$ and each $i=1,\cdots,n$, $$\left\lvert f_i(x)-f_i(y)\right\rvert<\gamma.$$ Fix $C$ large enough so that $\mu\left(\Lambda^{T_0}_{\eta_0}(C)\right)$ is arbitrarily close to $\mu(\Gamma^{T_0})$. For this fixed $C$, choose a positive integer $j_0=j_0(C)$ such that $C<e^{\frac{j_0T_0\epsilon}{2}}$. Then for every $x\in\Lambda^{T_0}_{\eta_0}(C)$ and every $J\geq 1$,  $$\prod_{i=0}^{J-1}\left\lVert\psi^*_{j_0T_0}|_{E^s\left(\varphi_{ij_0T_0}(x)\right)}\right\rVert\leq e^{-(\chi-\epsilon)Jj_0T_0},~\prod_{i=0}^{J-1}m\left(\psi^*_{j_0T_0}|_{E^u\left(\varphi_{ij_0T_0}(x)\right)}\right)\geq e^{(\chi-\epsilon)Jj_0T_0}.$$ Setting $T=j_0T_0$ and $\eta=(\chi-\epsilon)j_0T_0$, we consider the set
\begin{equation*}
\begin{aligned}
   \Lambda^{T}_{\eta}(C)=\Bigg\{x\in\Gamma:&\prod_{i=0}^{J-1}\left\lVert\psi^*_{T}|_{E^s\left(\varphi_{iT}(x)\right)}\right\rVert\leq e^{-J\eta},~\forall~ J\geq 1,\\ &\prod_{i=0}^{J-1}m\left(\psi^*_{T}|_{E^u\left(\varphi_{iT}(x)\right)}\right)\geq e^{J\eta},~\forall~ J\geq 1,~d(x,\text{Sing}(X))\geq\dfrac{1}{C}\Bigg\}.
\end{aligned}   
\end{equation*}	
   Take a point $x_0\in\Lambda^T_\eta(C)\cap\text{supp}(\mu)$. By the continuity of $X(x)$, there exists $r>0$ sufficiently small such that for every $x\in B(x_0,r)$, $$B(x_0,r)\subset B\left(x,\xi\left\lvert X(x)\right\rvert\right).$$ Let $\varepsilon'=\min\left\{\varepsilon,~\xi\right\}$. For every $\alpha\in(0,1/C)$, Theorem \ref{Shadow} provides a positive number $$\mathcal{D}=\mathcal{D}(\alpha,\varepsilon')>0.$$ Since $\mu\left(B(x_0,r)\cap\Lambda^T_\eta(C)\cap\text{supp}(\mu)\cap\Lambda_B\right)>0$, the Poincar\'{e} Recurrence Theorem implies that there exist a point $y\in B(x_0,r)\cap\Lambda^T_\eta(C)\cap\text{supp}(\mu)\cap\Lambda_B$ and a sufficiently large integer $l$ such that $$\varphi_{lT}(y)\in B(x_0,r)\cap\Lambda^T_\eta(C)\cap\text{supp}(\mu)\cap\Lambda_B\text{ and } d(y,\varphi_{lT}(y))<\mathcal{D}.$$ Thus, we obtain an $(\eta,T)$-$\psi^*_t$-quasi-hyperbolic orbit segment $\varphi_{[0,lT]}(y)$ satisfying
\begin{itemize}
   \item $d(y,\text{Sing}(X))>\alpha$ and $d(\varphi_{lT}(y),\text{Sing}(X))>\alpha${\rm;}
		
   \item $y\in\Lambda^T_\eta(C)$, $\varphi_{lT}(y)\in\Lambda^T_\eta(C)$ and $d(y,\varphi_{lT}(y))<\mathcal{D}${\rm.} 
\end{itemize}
   By Theorem \ref{Shadow}, there exist a strictly increasing $C^1$ function $\theta:[0,lT]\to\mathbb{R}$ and a periodic point $p\in M$ such that   
\begin{description}
   \item[(1)] $\theta(0)=0$ and $1-\gamma\leq1-\xi<\theta'(t)<1+\xi\leq1+\gamma$, for every $t\in[0,lT]${\rm;}
		
   \item[(2)] $p$ is a periodic point with period $\theta(lT)$ {\rm:} $\varphi_{\theta(lT)}(p)=p${\rm;}
		
   \item[(3)] $d(\varphi_t(y),\varphi_{\theta(t)}(p))<\varepsilon'\left\lvert X(\varphi_t(y))\right\rvert<\xi\left\lvert X(\varphi_t(y))\right\rvert$, for every $t\in[0,lT]${\rm.}
\end{description}
   Therefore, for each $i=1,2,\cdots,n$, we have 
\begin{equation*}
\begin{aligned}
   \left\lvert\dfrac{1}{\theta(lT)}\int_{0}^{\theta(lT)}f_i(\varphi_t(p))dt-\dfrac{1}{lT}\int_{0}^{\theta(lT)}f_i(\varphi_t(p))dt\right\rvert&\leq\left\lvert\dfrac{\theta(lT)}{lT}-1\right\rvert\cdot\dfrac{1}{\theta(lT)}\int_{0}^{\theta(lT)}f_i(\varphi_t(p))dt\\&\leq\gamma K.
\end{aligned}   
\end{equation*} 
   and    
\begin{equation*}
\begin{aligned}
   &\left\lvert\dfrac{1}{lT}\int_{0}^{\theta(lT)}f_i(\varphi_s(p))ds-\dfrac{1}{lT}\int_{0}^{lT}f_i(\varphi_t(y))dt\right\rvert=\dfrac{1}{lT}\left\lvert\int_{0}^{lT}f_i(\varphi_{\theta(t)}(p))d\theta(t)-\int_{0}^{lT}f_i(\varphi_t(y))dt\right\rvert\\\leq&\dfrac{1}{lT}\left\lvert\int_{0}^{lT}f_i(\varphi_{\theta(t)}(p))(\theta'(t)-1)dt\right\rvert+\dfrac{1}{lT}\left\lvert\int_{0}^{lT}\left[f_i(\varphi_{\theta(t)}(p))-f_i(\varphi_t(y))\right]dt\right\rvert\\\leq&\gamma K+\gamma.
\end{aligned}   
\end{equation*} 
   Thus, $$\left\lvert\dfrac{1}{\theta(lT)}\int_{0}^{\theta(lT)}f_i(\varphi_t(p))dt-\dfrac{1}{lT}\int_{0}^{lT}f_i(\varphi_t(y))dt\right\rvert<\dfrac{\varepsilon}{4n}\cdot\min_{1\leq i\leq n}\left\{2^i\left\lVert f_i\right\rVert\right\}.$$
	
   Denote by $\mu_p$ the invariant measure supported on the periodic orbit ${\rm Orb}(p)$. From the equations $(\vartriangle)$ and $(\vartriangle\vartriangle)$, we have    
\begin{equation*}
\begin{aligned}
   &d_{\mathcal{M}}(\mu,\mu_p)=\sum_{i=1}^{+\infty}\dfrac{\left\lvert \displaystyle\int_{M}f_id\mu-\int_{M}f_id\mu_p\right\rvert}{2^i\left\lVert f_i\right\rVert}=\sum_{i=1}^{n}\dfrac{\left\lvert \displaystyle\int_{M}f_id\mu-\int_{M}f_id\mu_p\right\rvert}{2^i\left\lVert f_i\right\rVert}+\sum_{i=n+1}^{+\infty}\dfrac{\left\lvert \displaystyle\int_{M}f_id\mu-\int_{M}f_id\mu_p\right\rvert}{2^i\left\lVert f_i\right\rVert}\\&\leq\sum_{i=1}^{n}\dfrac{\left\lvert \displaystyle\int_{M}f_id\mu-\dfrac{1}{lT}\int_{0}^{lT}f_i(\varphi_t(y))dt+\dfrac{1}{lT}\int_{0}^{lT}f_i(\varphi_t(y))dt-\int_{M}f_id\mu_p\right\rvert}{2^i\left\lVert f_i\right\rVert}+\sum_{i=n+1}^{+\infty}\dfrac{1}{2^{i-1}}\\&\leq\sum_{i=1}^{n}\dfrac{\left\lvert \displaystyle\int_{M}f_id\mu-\dfrac{1}{lT}\int_{0}^{lT}f_i(\varphi_t(y))dt\right\rvert+\left\lvert\dfrac{1}{lT} \displaystyle\int_{0}^{lT}f_i(\varphi_t(y))dt-\dfrac{1}{\theta(lT)}\int_{0}^{\theta(lT)}f_i(\varphi_t(p))dt\right\rvert}{2^i\left\lVert f_i\right\rVert}+\dfrac{\varepsilon}{2}\\&<\sum_{i=1}^{n}\dfrac{\dfrac{\varepsilon}{4n}\cdot\min\limits_{1\leq i\leq n}\left\{2^i\left\lVert f_i\right\rVert\right\}+\dfrac{\varepsilon}{4n}\cdot\min\limits_{1\leq i\leq n}\left\{2^i\left\lVert f_i\right\rVert\right\}}{2^i\left\lVert f_i\right\rVert}+\dfrac{\varepsilon}{2}\leq\sum_{i=1}^{n}\dfrac{\varepsilon}{2n}+\dfrac{\varepsilon}{2}<\varepsilon.
\end{aligned}   
\end{equation*}
   This completes the proof of Theorem \ref{ThmA}.          
\end{proof}

\section{Approximation of Lyapunov exponents: proof of Theorem \ref{ThmB}}\label{PF}   
   In this section, we prove Theorem \ref{ThmB}. Let $\varphi_t$ be the $C^1$ flow generated by a vector field $X\in\mathfrak{X}^1(M)$, and let $\mu$ be an ergodic hyperbolic invariant regular measure. Throughout, we assume that the Oseledec splitting $$\mathcal{N}_\Gamma=E_1\oplus E_2\oplus\cdots\oplus E_s\oplus E_{s+1}\oplus\cdots\oplus E_k$$ of the scaled linear Poincar\'{e} flow with respect to $\mu$ is dominated. When convenient, we denote this splitting as $$\mathcal{N}_\Gamma=E^s\oplus E^u,$$ where $E^s=E_1\oplus E_2\oplus\cdots\oplus E_s,\text{ and } E^u=E_{s+1}\oplus E_{s+2}\oplus\cdots\oplus E_k$ are called the stable and unstable bundles, respectively. Then $$\mathcal{N}_\Gamma=E^s\oplus E^u$$ is dominated. To prove Theorem \ref{ThmB}, we first establish estimates on the Lyapunov exponents of the scaled linear Poincar\'{e} flow $\psi^*_t$, stated as Propositions \ref{TLLE} and \ref{TSLE} below.         	  
\begin{Proposition}\label{TLLE}{\rm (The approximation of largest Lyapunov exponents)}
   Let $\mu$ be an ergodic hyperbolic invariant regular measure for a $C^1$ vector field $X\in\mathfrak{X}^1(M)$ on a compact smooth Riemannian manifold, and let $\lambda_1\leq\lambda_2\leq\cdots\leq\lambda_{d-1}$ denote the Lyapunov exponents of the scaled linear Poincar\'{e} flow $\psi^*_t$ with respect to $\mu$. If the Oseledec splitting $$\mathcal{N}_\Gamma=E_1\oplus E_2\oplus\cdots\oplus E_s\oplus E_{s+1}\oplus\cdots\oplus E_k$$ of the scaled linear Poincar\'{e} flow with respect to $\mu$ is dominated, then for every $\epsilon>0$, there exists a hyperbolic periodic point $p$ of the flow $\varphi_t$ such that $$\lvert\lambda_i-\lambda_i(p)\rvert<\epsilon,\quad\text{for each }i=d-d_k,d-d_k+1,\cdots,d-1,$$ where $d_k=\text{dim}(E_k)$, $\lambda_1(p)\leq\lambda_2(p)\leq\cdots\leq\lambda_{d-1}(p)$ are the Lyapunov exponents of the scaled linear Poincar\'{e} flow $\psi^*_t$ with respect to the periodic measure $\mu_p$ supported on ${\rm Orb}(p)$.     	
\end{Proposition}

\begin{proof}
   Let $$\chi=\min_{1\leq j\leq k}\{\lvert\chi_j\rvert\}.$$ Since the Oseledec splitting $\mathcal{N}_\Gamma=E^s\oplus E^u$ is dominated, it follows from Lemma \ref{time} that for every $0<\epsilon/4\ll\chi$ and every $\delta_1\in(0,1/6d)$, there exists a positive number $L=L(\epsilon/4,\delta_1)$ such that
\begin{itemize}
   \item for every $T\geq L$, there exists a measurable set $\Gamma^T=\Gamma^T(\epsilon/4,\delta_1)\subset\Gamma$ with $\mu(\Gamma^T)\geq 1-\delta_1$;
    	
   \item there exists a positive integer $N_1=N(T)$ such that for every integer $n\geq N_1$ and every $x\in\Gamma^T$, $$\prod_{i=0}^{n-1}\left\lVert\psi^*_T|_{E^s\left(\varphi_{iT}(x)\right)}\right\rVert\leq e^{-(\chi-\epsilon/4)nT},~\prod_{i=0}^{n-1}m\left(\psi^*_T|_{E^u\left(\varphi_{iT}(x)\right)}\right)\geq e^{(\chi-\epsilon/4)nT},$$ $$\text{ and }\dfrac{\left\lVert\psi^*_T|_{E^s}\right\rVert}{m\left(\psi^*_T|_{E^u}\right)}\leq e^{-T(\chi-\epsilon/4)}.$$     	  
\end{itemize}
   Fix an integer $T_0\geq L(\epsilon/4,\delta_1)$ and set $\eta_0=(\chi-\epsilon/4)T_0$. For every $C>0$, we define the {\bf Pesin block } $\Lambda^{T_0}_{\eta_0}(C)$ by
\begin{equation*}
\begin{aligned}
   \Lambda^{T_0}_{\eta_0}(C)=\Bigg\{x\in\Gamma:&\prod_{i=0}^{n-1}\left\lVert\psi^*_{T_0}|_{E^s\left(\varphi_{iT_0}(x)\right)}\right\rVert\leq Ce^{-n\eta_0},~\forall~ n\geq 1,\\ &\prod_{i=0}^{n-1}m\left(\psi^*_{T_0}|_{E^u\left(\varphi_{iT_0}(x)\right)}\right)\geq C^{-1}e^{n\eta_0},~\forall~ n\geq 1,~d(x,\text{Sing}(X))\geq\dfrac{1}{C}\Bigg\}.
\end{aligned}   
\end{equation*}
   According to \cite[Proposition 5.3]{WYZ}, $\Lambda^{T_0}_{\eta_0}(C)$ is compact and $$\mu\left(\Lambda^{T_0}_{\eta_0}(C)\right)\to\mu\left(\Gamma^{T_0}\right)\text{~~as~~}C\to+\infty.$$ Therefore, we can choose a sufficiently large constant $C>0$ such that $\mu\left(\Lambda^{T_0}_{\eta_0}(C)\right)>1-2\delta_1>0$. For this fixed $C$, there exists a positive integer $j_0=j_0(C)$ satisfying $C<e^{\frac{j_0T_0\epsilon}{4}}$. Consequently, for every $x\in\Lambda^{T_0}_{\eta_0}(C)$, $$\prod_{i=0}^{n-1}\left\lVert\psi^*_{j_0T_0}|_{E^s\left(\varphi_{ij_0T_0}(x)\right)}\right\rVert\leq e^{-(\chi-\epsilon/2)nj_0T_0},~\prod_{i=0}^{n-1}m\left(\psi^*_{j_0T_0}|_{E^u\left(\varphi_{ij_0T_0}(x)\right)}\right)\geq e^{(\chi-\epsilon/2)nj_0T_0},\text{ for }\forall~n\geq 1.$$ Setting $T=j_0T_0$ and $\eta=(\chi-\epsilon/2)j_0T_0$, we consider the set
\begin{equation*}
\begin{aligned}
   \Lambda^{T}_{\eta}(C)=\Bigg\{x\in\Gamma:&\prod_{i=0}^{n-1}\left\lVert\psi^*_{T}|_{E^s\left(\varphi_{iT}(x)\right)}\right\rVert\leq e^{-n\eta},~\forall~ n\geq 1,\\ &\prod_{i=0}^{n-1}m\left(\psi^*_{T}|_{E^u\left(\varphi_{iT}(x)\right)}\right)\geq e^{n\eta},~\forall~ n\geq 1,~d(x,\text{Sing}(X))\geq\dfrac{1}{C}\Bigg\}.
\end{aligned}   
\end{equation*}	 

   Given $\epsilon/4>0$, we can choose $\zeta>0$ such that $$\left\lvert\log(1+\zeta)\right\rvert<\dfrac{\epsilon}{4},\text{ and }\left\lvert\log(1-\zeta)\right\rvert<\dfrac{\epsilon}{4}.$$ Recall that the extended linear Poincar\'{e} flow $\widetilde{\psi}_t(v)$ varies continuously with respect to the vector field $X$, the time $t$ and the vector $v$. For a fixed $T>0$, since $\widetilde{\psi}^*_T$ and its inverse $\left(\widetilde{\psi}^*_T\right)^{-1}$ are both continuous on $$\widetilde{M}=\text{Closure}\left(\bigcup_{x\in M\backslash \text{Sing}(X)}\dfrac{X(x)}{\lvert X(x)\rvert}\right),$$ the map $$(x_1,x_2)\mapsto{\rm I}(x_1,x_2)=\left(\psi^*_T|_{\mathcal{N}_{x_1}}\right)^{-1}\circ\left(\psi^*_T|_{\mathcal{N}_{x_2}}\right)$$ is continuous for every $x_1,~x_2\in M\backslash\text{Sing}(X)$. Consequently, $${\rm I}(x_1,x_2)=\left(\psi^*_T|_{\mathcal{N}_{x_1}}\right)^{-1}\circ\left(\psi^*_T|_{\mathcal{N}_{x_2}}\right)$$ approaches the identity map $$\left(\psi^*_T|_{\mathcal{N}_{\varphi_T(x_2)}}\right)^{-1}\circ\left(\psi^*_T|_{\mathcal{N}_{x_2}}\right)={\rm I}$$ when both the distance between $x_1$ and $\varphi_T(x_2)$, and the distance between $\varphi_{-T}(x_1)$ and $x_2$ are sufficiently small. Let $\sigma_1>0$ be chosen such that $$\max\{d(x_1,\varphi_T(x_2)),d(\varphi_{-T}(x_1),x_2)\}\leq\sigma_1\Rightarrow 1-\zeta\leq\lVert {\rm I}(x_1,x_2)\rVert\leq1+\zeta.$$ Hence, $$\lvert\log\lVert {\rm I}(x_1,x_2)\rVert\rvert<\dfrac{\epsilon}{4}.$$ Such a $\sigma_1$ exists because $\widetilde{M}$ is compact. For a fixed $T>0$, the norm $$\lVert\psi^*_T\rVert=\sup\left\{\lvert\psi^*_T(v)\rvert:~v\in\mathcal{N},~\lvert v\rvert=1\right\}$$ is uniformly upper bounded on $\mathcal{N}$, and the co-norm $$m\left(\psi^*_T\right)=\inf\left\{\lvert\psi^*_T(v)\rvert:~v\in\mathcal{N},~\lvert v\rvert=1\right\}$$ is uniformly bounded away from $0$ on $\mathcal{N}$. Since the flow $\varphi_t$ is $C^1$, there exists $\sigma'_2>0$ such that for any $t_1,~t_2\in\mathbb{R}$ with $\lvert t_1-t_2\rvert<\sigma'_2$, $$d(\varphi_{t_1}(x),\varphi_{t_2}(x))<\dfrac{\sigma_1}{2},\text{ for every }x\in M.$$ Moreover, for the $\zeta>0$ chosen above, there exists $\sigma''_2>0$ such that for every $s\in(T-\sigma''_2,T+\sigma''_2)$ and every $x\in M\backslash\text{Sing}(X)$, $$\left\lvert\left\lVert\left(\psi^*_T|_{\mathcal{N}_{\varphi_s(x)}}\right)^{-1}\circ\psi^*_s|_{\mathcal{N}_x}\right\rVert-\left\lVert\left(\psi^*_s|_{\mathcal{N}_{\varphi_s(x)}}\right)^{-1}\circ\psi^*_s|_{\mathcal{N}_x}\right\rVert\right\rvert\leq\zeta.$$ Set $\sigma_2=\min\{\sigma'_2,\sigma''_2\}$. Then for every $s$ with $\left\lvert s-T\right\rvert<\sigma_2$, we obtain $$\left\lvert\log\left\lVert\left(\psi^*_T|_{\mathcal{N}_{\varphi_s(x)}}\right)^{-1}\circ\psi^*_s|_{\mathcal{N}_x}\right\rVert\right\rvert<\dfrac{\epsilon}{4}.$$
   
   For estimating the Lyapunov exponents, we state the following claim concerning invertible linear maps.   
\begin{claim}
   Let $A,~B:\mathbb{R}^d\to\mathbb{R}^d$ be two invertible linear maps. For every $\epsilon_1>0$, there exists $\sigma_3>0$ such that if $\lVert B-I\rVert\leq\sigma_3$, then
\begin{equation*}
   \lVert ABv\rVert\geq e^{-\epsilon_1}\lVert Av\rVert,\quad \forall~v\in\mathbb{R}^d.\tag{\S}
\end{equation*}   
\end{claim} 

\begin{proof}[Proof of the Claim]
   It suffices to prove the inequality $(\S)$ for unit vectors $v\in\mathbb{R}^n$ with $\lVert v\rVert=1$. Fix $\epsilon_1>0$ and choose $\epsilon_2>0$ such that $1-\epsilon_2>e^{-\epsilon_1}$. Define $$\sigma_3=\dfrac{m(A)}{\lVert A\rVert}\epsilon_2,$$ where $m(A)=\inf\limits_{\lVert\omega\rVert=1}\lVert A\omega\rVert$ denotes the co-norm of $A$. Then, $$\lVert A\rVert\cdot\lVert B-I\rVert\leq\epsilon_2m(A),$$ whenever $\lVert B-I\rVert\leq\sigma_3$. Now assume $\lVert B-I\rVert\leq\sigma_3$.  For every unit vector $v$,
\begin{equation*}
\begin{aligned}
   \lVert ABv\rVert&=\lVert Av+A(B-I)v\rVert\geq \lVert Av\rVert-\lVert A\rVert\cdot\lVert B-I\rVert\cdot\lVert v\rVert=\lVert Av\rVert-\lVert A\rVert\cdot\lVert B-I\rVert\\&\geq\lVert Av\rVert-\epsilon_2m(A)\geq(1-\epsilon_2)\lVert Av\rVert\geq e^{-\epsilon_1}\lVert Av\rVert.
\end{aligned}
\end{equation*}      
\end{proof}   
   
   Take a point $x\in M$ and a sufficiently small neighborhood $U(x)$, similarly to the technique used in \cite{PM93}, by parallel transporting vectors along the unique geodesic in the Sasaki metric, we can locally trivialize the tangent bundle $T_{U(x)}M=U(x)\times\mathbb{R}^d$ of $U(x)$. Consequently, for every $y\in U(x)$, we can parallel translate vectors from $T_xM$ to $T_yM$, and hence also translate vectors from $\mathcal{N}_x$ to $\mathcal{N}_y$. According to Lemma \ref{Dlgx}, there exists a constant $\beta^*_0>0$ such that for every $x\in\Lambda^{T}_{\eta}(k)$, every point $y\in U(x)=B(x,\beta^*_0\lvert X(x)\rvert)$ is not a singularity. For every $\rho>1$, we define the $\rho$-cone in $\mathcal{N}_y$ by $$\mathcal{C}_\rho(y)=\{v\in \mathcal{N}_y:~\lVert v_k\rVert\geq\rho\lVert v_j\rVert,~j=1,2,\cdots,k-1\},$$ where $v=v_1+v_2+\cdots+v_k$ with $v_j\in E_j(x)$ for $j=1,2,\cdots,k$, and $E_j(x)$ denotes the Oseledec subspace. In other words, after translating the splitting $\mathcal{N}_x=E_1\oplus E_2\oplus\cdots\oplus E_k$ to $\mathcal{N}_y$, the vector $v\in\mathcal{N}_y$ decomposes correspondingly as $v=v_1+v_2+\cdots+v_k$. For computational purposes, we can define an equivalent norm $\lVert\cdot\rVert'$ on $\mathcal{N}_y$ by $$\lVert v\rVert'=\max\{\lVert v_1\rVert,\lVert v_2\rVert,\cdots,\lVert v_k\rVert\}.$$ Since the Oseledec splitting $$\mathcal{N}_\Gamma=E_1\oplus E_2\oplus\cdots\oplus E_k$$ is dominated for $\psi^*_T$, there exist two positive numbers $\gamma\in(0,1)$ and $\rho>1$ such that $\gamma\rho>1$ and $$\psi^*_{T}\mathcal{C}_{\gamma\rho}(x)\subset\mathcal{C}_{\rho}(\varphi_{T}(x)),\quad\forall~x\in\Lambda^{T}_{\eta}(C).$$ Furthermore, we can choose $\sigma_4=\sigma_4(\gamma,\rho)>0$ such that for every linear isomorphism $\mathcal{T}$, $$\lVert \mathcal{T}-{\rm I}\rVert\leq\sigma_4\Longrightarrow \mathcal{T}\mathcal{C}_{\rho}\subset\mathcal{C}_{\gamma\rho}.$$   
   
   Take $y\in\Lambda^{T}_{\eta}(C)\cap{\rm supp}(\mu)$. By Poincar\'{e} Recurrence Theorem, there exists an increasing sequence of integers  $\{l_n\}$ such that $$d(y,\varphi_{l_nT}(y))\rightarrow 0 \text{ as }n\rightarrow+\infty.$$ Set $$\sigma_3=\dfrac{m\left(\psi^*_T\right)}{\left\lVert\psi^*_T\right\rVert}\left(1-e^{-\epsilon/4}\right).$$ For $0<\varepsilon<\min\left\{\beta^*_0,\xi,\dfrac{r_0}{K_0},\dfrac{\epsilon}{4},\dfrac{\sigma_1}{2K_0},\dfrac{\sigma_2}{K_0},\sigma_3,\log\left(1+\sigma_4\right)\right\}$ and $\alpha\in(0,1/C)$, Theorem \ref{Shadow} provides a constant $\mathcal{D}=\mathcal{D}(\varepsilon,\alpha)>0$. For sufficiently large $l_n$, we can obtain a $(\eta,T)$-$\psi^*_t$-quasi-hyperbolic orbit segment $\varphi_{[0,l_nT]}(y)$ satisfying:
\begin{itemize}
   \item $d(y,\text{Sing}(X))>\alpha$ and $d(\varphi_{l_nT}(y),\text{Sing}(X))>\alpha${\rm;}
   	
   \item $y\in\Lambda^{T}_{\eta}(C)$, $\varphi_{l_nT}(y)\in\Lambda^{T}_{\eta}(C)$ and $d(y,\varphi_{l_nT}(y))<\mathcal{D}${\rm.} 
\end{itemize}
   By Theorem \ref{Shadow}, there exist a strictly increasing $C^1$ function $\theta:[0,l_nT]\to\mathbb{R}$ and a periodic point $p\in M$ such that
\begin{itemize}
   \item[(1)] $\theta(0)=0$ and $1-\varepsilon<\theta'(t)<1+\varepsilon$, for every $t\in[0,l_nT]${\rm;}
   	
   \item[(2)] $p$ is a periodic point with period $\theta(l_nT)$ {\rm:} $\varphi_{\theta(l_nT)}(p)=p${\rm;}
   	
   \item[(3)] $d(\varphi_t(y),\varphi_{\theta(t)}(p))<\varepsilon\lvert X(\varphi_t(y))\rvert$, for every $t\in[0,l_nT]${\rm.}
\end{itemize}
   Moreover, by Proposition 4.4 in \cite{WYZ}, there exists a constant $N=N(\eta,T)$ such that $$\left\lvert\theta(iT)-iT\right\rvert\leq Nd(y,\varphi_{l_nT}(y)),\text{ for }i=1,2,\cdots,l_n.$$ Selecting $l_n$ large enough so that $$d(y,\varphi_{l_nT}(y))<\dfrac{\sigma_2}{2N},$$ and writing simply $l=l_n$ when no confusion arises, we have $$\dfrac{1}{lT}\sum_{i=0}^{l-1}\log\left\lVert\psi^*_{T}|_{\mathcal{N}_{\varphi_{iT}(y)}}\right\rVert\leq\lambda_{d-1}+\dfrac{\epsilon}{2}.$$ The point $p$ is periodic with period $\theta(lT)$. Associated with this orbit is the invariant periodic measure $$\mu_p=\dfrac{1}{\theta(lT)}\int_{0}^{\theta(lT)}\delta_{\varphi_t(p)}dt,$$ supported on ${\rm Orb}(p)$. Let $$\lambda_1(p)\leq\lambda_2(p)\leq\cdots\leq\lambda_{d-1}(p)$$ denote the Lyapunov exponents of the scaled linear Poincar\'{e} flow $\psi^*_t$ with respect to $\mu_p$. In order to compute these exponents, we first establish the following claim.
\begin{claim}
   For each $i=0,1,\cdots,l-1$, define $$T_i(y,p)=\left(\psi^*_T|_{\mathcal{N}_{\varphi_{(i+1)T}(y)}}\right)^{-1}\circ\psi^*_{\theta((i+1)T)-\theta(iT)}|_{\mathcal{N}_{\varphi_{\theta(iT)}(p)}}.$$ Then $$\left\lvert\log\left\lVert T_i(y,p)\right\rVert\right\rvert<\dfrac{\epsilon}{2}.$$  	
\end{claim}

\begin{proof}[Proof of the Claim]
   For each $i=0,1,\cdots,l-1$, the composition  $$\psi^*_T|_{\mathcal{N}_{\varphi_{\theta((i+1)T)-T}(p)}}\circ\left(\psi^*_T|_{\mathcal{N}_{\varphi_{\theta((i+1)T)}(p)}}\right)^{-1}$$ equals the identity. Hence, we consider the operator $$\left(\psi^*_T|_{\mathcal{N}_{\varphi_{(i+1)T}(y)}}\right)^{-1}\circ\psi^*_T|_{\mathcal{N}_{\varphi_{\theta((i+1)T)-T}(p)}}\circ\left(\psi^*_T|_{\mathcal{N}_{\varphi_{\theta((i+1)T)}(p)}}\right)^{-1}\circ\psi^*_{\theta((i+1)T)-\theta(iT)}|_{\mathcal{N}_{\varphi_{\theta(iT)}(p)}}.$$ Since $$\lvert\theta(iT)-iT\rvert\leq Nd(y,\varphi_{lT}(y)),\quad\text{for }i=0,1,2,\cdots,l,$$ we have
\begin{equation*}
\begin{aligned}
   &\left\lvert\theta((i+1)T)-T-\theta(iT)\right\rvert= \left\lvert\theta((i+1)T)-\theta(iT)-((i+1)T-iT)\right\rvert\\=&\left\lvert\theta((i+1)T)-(i+1)T-(\theta(iT)-iT)\right\rvert\leq\left\lvert\theta((i+1)T)-(i+1)T\right\rvert+\left\lvert(\theta(iT)-iT)\right\rvert\\\leq&2Nd(y,\varphi_{lT}(y))<\sigma_2.
\end{aligned}  
\end{equation*}
   Consequently, $\theta((i+1)T)-\theta(iT)\in(T-\sigma_2,T+\sigma_2)$, for vevery $i=0,1,\cdots,l-1$. Therefore, $$d(\varphi_{\theta((i+1)T)-T}(p),\varphi_{\theta(iT)}(p))<\dfrac{\sigma_1}{2},$$ and $$\left\lvert\log\left\lVert\left(\psi^*_T|_{\mathcal{N}_{\varphi_{\theta((i+1)T)}(p)}}\right)^{-1}\circ\psi^*_{\theta((i+1)T)-\theta(iT)}|_{\mathcal{N}_{\varphi_{\theta(iT)}(p)}}\right\rVert\right\rvert<\dfrac{\epsilon}{4},\text{ for every $i=0,1,\cdots,l-1$}.$$ On the other hand, since $$d(\varphi_t(y),\varphi_{\theta(t)}(p))<\varepsilon\lvert X(\varphi_t(y))\rvert,\text{ for every }t\in[0,lT],$$ we obtain $$d(\varphi_{iT}(y),\varphi_{\theta(iT)}(p))<\varepsilon\lvert X(\varphi_{iT}(y))\rvert<\sigma_1,\quad\text{for every }i=0,1,2,\cdots,l.$$ Thus, $$\left\lvert\log\left\lVert\left(\psi^*_T|_{\mathcal{N}_{\varphi_{(i+1)T}(y)}}\right)^{-1}\circ\psi^*_T|_{\mathcal{N}_{\varphi_{\theta((i+1)T)-T}(p)}}\right\rVert\right\rvert<\dfrac{\epsilon}{4}.$$ Combining these estimates, we conclude that $$\left\lvert\log\left\lVert T_i(y,p)\right\rVert\right\rvert<\dfrac{\epsilon}{2}.$$       	
\end{proof}
   Since $$\psi^*_{\theta(lT)}|_{\mathcal{N}_p}=\psi^*_T|_{\mathcal{N}_{\varphi_{(l-1)T}(y)}}\circ T_{l-1}(y,p)\circ\cdots\circ\psi^*_T|_{\mathcal{N}_{\varphi_T(y)}}\circ T_1(y,p)\circ\psi^*_T|_{\mathcal{N}_y}\circ T_0(y,p),$$ we have that 
\begin{equation*}
\begin{aligned}
   \lambda_{d-1}(p)&=\lim_{J\rightarrow+\infty}\dfrac{1}{J\theta(lT)}\log\left\lVert\psi^*_{J\theta(lT)}|_{\mathcal{N}_p}\right\rVert=\lim_{J\rightarrow+\infty}\dfrac{1}{J\theta(lT)}\log\left\lVert \left(\psi^*_{\theta(lT)}|_{\mathcal{N}_p}\right)^J\right\rVert\\&\leq\dfrac{1}{\theta(lT)}\log\left\lVert\psi^*_{\theta(lT)}|_{\mathcal{N}_p}\right\rVert\leq\dfrac{1}{\theta(lT)}\left(\sum_{i=0}^{l-1}\log\left\lVert\psi^*_T|_{\mathcal{N}_{\varphi_{iT}(y)}}\right\rVert+\sum_{i=0}^{l-1}\log\left\lVert T_i(y,p)\right\rVert\right).
\end{aligned}
\end{equation*}
   Because $1-\varepsilon<\theta'(t)<1+\varepsilon$, for every $t\in[0,lT]$, $$(1-\varepsilon)lT<\lvert\theta(lT)\rvert<(1+\varepsilon)lT.$$ Hence, $$\lambda_{d-1}(p)\leq\dfrac{lT(1+\varepsilon)}{\theta(lT)}\cdot\dfrac{1}{lT}\left(\sum_{i=0}^{l-1}\log\left\lVert\psi^*_T|_{\mathcal{N}_{\varphi_{iT}(y)}}\right\rVert+\sum_{i=0}^{l-1}\log\left\lVert T_i(y,p)\right\rVert\right)\leq\lambda_{d-1}+\epsilon.$$
   
   Since $\varepsilon<\log(1+\sigma_4)$, by an argument similar to that in the claim, we have $$\lVert T_i(y,p)-{\rm I}\rVert\leq\sigma_4,\text{ for every }i=0,1,\cdots,l-1$$ and $$\lVert T_i(y,p) v\rVert'\geq e^{-\frac{\epsilon}{2}}\lVert v\rVert',\quad\forall~ v\in\mathcal{N}_{\varphi_{\theta(iT)}(p)}\big\backslash\{0\},\text{ for every }i=0,1,\cdots,l-1.$$ Consequently, for each $i=0,1,\cdots,l-1$,
\begin{equation*}
   \psi^*_{T}|_{\mathcal{N}_{\varphi_{iT}(y)}}\circ T_i(y,p)\mathcal{C}_{\rho}\left(\varphi_{\theta(iT)}(p)\right)\subset\mathcal{C}_{\rho}\left(\varphi_{\theta((i+1)T)}(p)\right).\tag{$\vartriangle$}
\end{equation*}
   For every $v=(v_1,v_2,\cdots,v_k)\in\mathcal{C}_\rho(p)$ and each $i=0,1,\cdots,l-1$, define $$\bar{v}^0=v,\quad \bar{v}^i=T_i(y,p)\circ\psi^*_{T}|_{\mathcal{N}_{\varphi_{(i-1)T}(y)}}\circ T_{i-1}(y,p)\circ\cdots\circ\psi^*_{T}|_{\mathcal{N}_y}\circ T_0(y,p)(v)$$ and $$v^{i+1}=\psi^*_{T}|_{\mathcal{N}_{\varphi_{iT}(y)}}\circ T_i(x,p)\circ\psi^*_{T}|_{\mathcal{N}_{\varphi_{(i-1)T}(y)}}\circ T_{i-1}(x,p)\circ\cdots\circ\psi^*_{T}|_{\mathcal{N}_y}\circ T_0(y,p)(v).$$ Corresponding to the Oseledec splitting $$\mathcal{N}_\Gamma=E_1\oplus E_2\oplus\cdots\oplus E_k,$$ write $$\bar{v}^i=\bar{v}^i_1+\bar{v}^i_2+\cdots+\bar{v}^i_k,\quad v^i=v^i_1+v^i_2+\cdots+v^i_k.$$ Then, for each $i=0,1,\cdots,l-1$, we have
\begin{equation}
\begin{aligned}
   \left\lVert v^{i+1}\right\rVert'=\left\lVert v^{i+1}_k\right\rVert&=\left\lVert\psi^*_T(\bar{v}^i_k)\right\rVert\geq m\left(\psi^*_T|_{E_k(\varphi_{iT}(y))}\right)\left\lVert\bar{v}^i_k\right\rVert=m\left(\psi^*_{T}|_{E_k(\varphi_{iT}(y))}\right)\left\lVert\bar{v}^i\right\rVert'\\&=m\left(\psi^*_{T}|_{E_k(\varphi_{iT}(y))}\right)\left\lVert T_i(x,p)v^i\right\rVert'\geq e^{-\frac{\epsilon}{2}}m\left(\psi^*_{T}|_{E_k(\varphi_{iT}(y))}\right)\left\lVert v^i\right\rVert'\\&\geq e^{-\frac{\epsilon}{2}(i+1)}\prod^{i}_{j=0}m\left(\psi^*_{T}|_{E_k(\varphi_{jT}(y))}\right)\lVert v\rVert'.
\end{aligned}
\end{equation}   
   Because $\mathcal{C}_\rho(p)$ contains a subspace of dimension $d_k$, the estimates in $(\Delta)$, $(1)$ together with Lemma \ref{time} imply that there exists a subset $H\subset\{1,2,\cdots,d-1\}$ with $\sharp(H)=d_k$ such that $$\lambda_i(p)\geq\lambda_{d-1}-\epsilon,\quad\text{for each }i\in H.$$ Since the Lyapunov exponents of $\mu_p$ are arranged in ascending order, $$\lambda_i(p)\geq\lambda_i-\epsilon,\quad\text{for each $i=d-d_k,d-d_k+1,\cdots,d-1$}.$$ This completes the proof of the proposition.
\end{proof}

\begin{Proposition}\label{TSLE}{\rm (The approximation of smallest Lyapunov exponents)}
   Let $\mu$ be an ergodic hyperbolic invariant regular measure for a $C^1$ vector field $X\in\mathfrak{X}^1(M)$ on a compact smooth Riemannian manifold, and let $\lambda_1\leq\lambda_2\leq\cdots\leq\lambda_{d-1}$ denote the Lyapunov exponents of the scaled linear Poincar\'{e} flow $\psi^*_t$ with respect to $\mu$. If the Oseledec splitting $$\mathcal{N}_\Gamma=E_1\oplus E_2\oplus\cdots\oplus E_s\oplus E_{s+1}\oplus\cdots\oplus E_k$$ of the scaled linear Poincar\'{e} flow with respect to $\mu$ is dominated, then for every $\epsilon>0$, there exists a hyperbolic periodic point $p$ of the flow $\varphi_t$ such that $$\left\lvert\lambda_i-\lambda_i(p)\right\rvert<\epsilon,\text{ for each }i=1,2,\cdots,d_1,$$ where $d_1=\text{dim}(E_1)$, $\lambda_1(p)\leq\lambda_2(p)\leq\cdots\leq\lambda_{d-1}(p)$ are the Lyapunov exponents of the scaled linear Poincar\'{e} flow $\psi^*_t$ with respect to the periodic measure $\mu_p$ supported on ${\rm Orb}(p)$.    	
\end{Proposition}

\begin{proof}
   For the vector field $X\in\mathfrak{X}^1(M)$, we consider the vector field $-X$. Let $\overline{\varphi}_t$ be the $C^1$ flow generated by the vector field $-X$. Correspondingly, we have the tangent flow $\overline{\Phi}_t$, the linear Poincar\'{e} flow $\overline{\psi}_t$, the extended linear Poincar\'{e} flow $\overline{\widetilde{\psi}}_t$ and the scaled linear Poincar\'{e} flow $\overline{\psi}^*_t$.
   
   Through the discussion in Section \ref{LEF}, we have that $\mu$ is an ergodic hyperbolic invariant regular measure for the flow $\overline{\varphi}_t$. Consider the scaled linear Poincar\'e flow $\overline{\psi}_t^*:\mathcal{N}\rightarrow\mathcal{N}$. By the Oseledec Theorem \cite{OV}, for $\mu$-almost every $x\in M$, there exist a positive integer $k\in[1,d-1]$, real numbers $\overline{\chi}_1<\overline{\chi}_2<\cdots<\overline{\chi}_k$ and a measurable $\overline{\psi}^*_t$-invariant splitting (for simplicity, we omit the base point)   $$\mathcal{N}=\overline{E}_1\oplus\overline{E}_2\oplus\cdots\oplus \overline{E}_k$$ of the normal bundle associated with the vector field $-X$, such that $$\lim_{t\to\pm\infty}\dfrac{1}{t}\log\left\lVert\overline{\psi}^*_t(v)\right\rVert=\overline{\chi}_i,~\forall~v\in\overline{E}_i,~v\neq 0,~i=1,2,\cdots,k.$$ Let $$\overline{d}_i={\rm dim}(\overline{E}_i),~i=1,2,\cdots,k$$ denote the multiplicities of those Lyapunov exponents. Ignoring multiplicity, we rewrite the Lyapunov exponents of the scaled linear Poincar\'{e} flow $\overline{\psi}_t^*$ with respect to $\mu$ as $\overline{\lambda}_1\leq\overline{\lambda}_2\leq\cdots\leq\overline{\lambda}_{d-1}$. In other words, $$\overline{\lambda}_j=\overline{\chi}_i,\text{ for any } \overline{d}_1+\overline{d}_2+\cdots+\overline{d}_{i-1}<j\leq \overline{d}_1+\overline{d}_2+\cdots+\overline{d}_i.$$ From the relation between the vector field $-X$ and the vector field $X$, we obtain $$\overline{\chi}_i=-\chi_{k-i+1},\quad\overline{E}_i=E_{k-i+1},\quad\overline{d}_i=\text{dim}\overline{E}_i=d_{k-i+1},\text{ for every }i=1,2,\cdots,k,$$ where $\chi_i$ $(i=1,2,\cdots,k)$ are the Lyapunov exponents of the scaled linear Poincar\'{e} flow $\psi^*_t$ with respect to $\mu$, and $E_i$ are the corresponding subspaces in its Oseledec splitting, with dimensions $d_i={\rm dim}(E_i)$. Ignoring multiplicity, let $\lambda_1\leq\lambda_2\leq\cdots\leq\lambda_{d-1}$ denote the Lyapunov exponents of the scaled linear Poincar\'{e} flow $\psi_t^*$ with respect to $\mu$. Then $$\overline{\lambda}_i=-\lambda_{d-i},~\text{ for }i=1,2,\cdots,d-1.$$ Furthermore, the Oseledec splitting for $\overline{\psi}^*_t$ inherits the dominated property from that for $\psi^*_t$. More precisely, since the splitting $$\mathcal{N}_\Gamma=E_1\oplus E_2\oplus\cdots\oplus E_s\oplus E_{s+1}\oplus\cdots\oplus E_k$$ for $\psi^*_t$ (with respect to $\mu$) is dominated, the corresponding splitting $$\mathcal{N}_\Gamma=\overline{E}_1\oplus\overline{E}_2\oplus\cdots\oplus \overline{E}_k$$ for $\overline{\psi}^*_t$ (with respect to $\mu$) is also dominated. Therefore, the hyperbolic Oseledec splitting $$\mathcal{N}_\Gamma=\overline{E}^s\oplus\overline{E}^u,\text{ where $\overline{E}^s=\overline{E}_1\oplus\overline{E}_2\oplus\cdots\oplus \overline{E}_s, \text{ and }\overline{E}^u=\overline{E}_{s+1}\oplus \overline{E}_{s+2}\oplus\cdots\oplus\overline{E}_k,$}$$ is a dominated splitting for $\overline{\psi}^*_t$ with respect to $\mu$. 
   
   Let $$\overline{\chi}=\min_{1\leq j\leq k}\{\lvert\overline{\chi}_j\rvert\}.$$ Since $\mathcal{N}_\Gamma=\overline{E}^s\oplus\overline{E}^u$ is a dominated splitting, by Lemma \ref{time}, for every $0<\epsilon/4\ll\overline{\chi}$ and every $\delta_2\in(0,1/6d)$, there exists $\overline{L}=\overline{L}(\epsilon/4,\delta_2)$ such that   
\begin{itemize}
   \item for every $T\geq\overline{L}$, there exists a measurable set $\overline{\Gamma}^T=\overline{\Gamma}^T(\epsilon/4,\delta_2)\subset\Gamma$ with $\mu\left(\overline{\Gamma}^T\right)\geq 1-\delta_2${\rm;}
	
   \item there exists a natural number $\overline{N}_1=\overline{N}(T)$ such that for every integer $n\geq\overline{N}_1$ and every $x\in\overline{\Gamma}^T$, $$\prod_{i=0}^{n-1}\left\lVert\overline{\psi}^*_T|_{\overline{E}^s\left(\overline{\varphi}_{iT}(x)\right)}\right\rVert\leq e^{-(\overline{\chi}-\epsilon/4)nT},~\prod_{i=0}^{n-1}m\left(\overline{\psi}^*_T|_{\overline{E}^u\left(\overline{\varphi}_{iT}(x)\right)}\right)\geq e^{(\overline{\chi}-\epsilon/4)nT},$$ $$\text{ and }\dfrac{\left\lVert\overline{\psi}^*_T|_{\overline{E}^s}\right\rVert}{m\left(\overline{\psi}^*_T|_{\overline{E}^u}\right)}\leq e^{-T(\overline{\chi}-\epsilon/4)}.$$ 
\end{itemize}     	
   Fix an integer $T_0\geq\overline{L}(\epsilon/4,\delta_2)$ and set $\eta_0=(\overline{\chi}-\epsilon/4)T_0$. For every $C>0$, we define the {\bf Pesin block } $\Lambda^{T_0}_{\eta_0}(C)$ for the scaled linear Poincar\'{e} flow $\overline{\psi}^*_t$ by
\begin{equation*}
\begin{aligned}
   \Lambda^{T_0}_{\eta_0}(C)=\Bigg\{x\in\Gamma:&\prod_{i=0}^{n-1}\left\lVert\overline{\psi}^*_{T_0}|_{\overline{E}^s\left(\overline{\varphi}_{iT_0}(x)\right)}\right\rVert\leq Ce^{-n\eta_0},~\forall~n\geq 1,\\ &\prod_{i=0}^{n-1}m\left(\overline{\psi}^*_{T_0}|_{\overline{E}^u\left(\overline{\varphi}_{iT_0}(x)\right)}\right)\geq C^{-1}e^{n\eta_0},~\forall~ n\geq 1,~d(x,\text{Sing}(-X))\geq\dfrac{1}{C}\Bigg\}.
\end{aligned}   
\end{equation*}
   According to \cite[Proposition 5.3]{WYZ}, $\Lambda^{T_0}_{\eta_0}(C)$ is compact and $$\mu\left(\Lambda^{T_0}_{\eta_0}(C)\right)\to\mu\left(\overline{\Gamma}^{T_0}\right)\text{ as } C\to+\infty.$$ Therefore, we can choose $C$ sufficiently large such that $\mu\left(\Lambda^{T_0}_{\eta_0}(C)\right)>1-2\delta_2>0$. For this fixed $C$, there exists a positive integer $j_0=j_0(C)$ such that $C<e^{\frac{j_0T_0\epsilon}{4}}$. Thus, for every $x\in\Lambda^{T_0}_{\eta_0}(C)$, $$\prod_{i=0}^{n-1}\left\lVert\overline{\psi}^*_{j_0T_0}|_{\overline{E}^s\left(\overline{\varphi}_{ij_0T_0}(x)\right)}\right\rVert\leq e^{-(\overline{\chi}-\epsilon/2)nj_0T_0},~\prod_{i=0}^{n-1}m\left(\overline{\psi}^*_{j_0T_0}|_{\overline{E}^u\left(\overline{\varphi}_{ij_0T_0}(x)\right)}\right)\geq e^{(\overline{\chi}-\epsilon/2)nj_0T_0},\text{ for all } n\geq 1.$$ Setting $\eta=(\overline{\chi}-\epsilon/2)j_0T_0$ and $T=j_0T_0$, we consider the set
\begin{equation*}
\begin{aligned}
   \Lambda^{T}_{\eta}(C)=\Bigg\{x\in\Gamma:&\prod_{i=0}^{n-1}\left\lVert\overline{\psi}^*_{T}|_{\overline{E}^s\left(\overline{\varphi}_{iT}(x)\right)}\right\rVert\leq e^{-n\eta},~\forall~n\geq 1,\\ &\prod_{i=0}^{n-1}m\left(\overline{\psi}^*_{T}|_{\overline{E}^u\left(\overline{\varphi}_{iT}(x)\right)}\right)\geq e^{n\eta},~\forall~ n\geq 1,~d(x,\text{Sing}(-X))\geq\dfrac{1}{C}\Bigg\}.
\end{aligned}   
\end{equation*}       

   Take $y\in\Lambda^{T}_{\eta}(C)\cap{\rm supp}(\mu)$. By Poincar\'{e} Recurrence Theorem, there exists an increasing sequence of integers  $\{l_n\}$ such that $$d(y,\overline{\varphi}_{l_nT}(y))\rightarrow 0\text{ as }n\rightarrow+\infty.$$ Choose appropriate parameters $\sigma_1,\sigma_2,\sigma_3,\sigma_4$ as in Proposition \ref{TLLE}. For $$0<\varepsilon<\min\left\{\beta^*_0,\xi,\dfrac{r_0}{K_0},\dfrac{\epsilon}{4},\dfrac{\sigma_1}{2K_0},\dfrac{\sigma_2}{K_0},\sigma_3,\log\left(1+\sigma_4\right)\right\}\text{ and }\alpha\in(0,1/C),$$ we obtain a constant $\mathcal{D}=\mathcal{D}(\varepsilon,\alpha)$ from Theorem \ref{Shadow}. For $l_n$ large enough, we obtain a $(\eta,T)$-$\overline{\psi}^*_t$-quasi-hyperbolic orbit segment $\overline{\varphi}_{[0,l_nT]}(y)$ satisfying that
\begin{itemize}
   \item $d(y,\text{Sing}(-X))>\alpha$ and $d(\overline{\varphi}_{l_nT}(y),\text{Sing}(-X))>\alpha${\rm;}
   	
   \item $y\in\Lambda^{T}_{\eta}(C)$, $\overline{\varphi}_{l_nT}(y)\in\Lambda^{T}_{\eta}(C)$ and $d(y,\overline{\varphi}_{l_nT}(y))<\mathcal{D}${\rm.}
\end{itemize}
   According to Theorem \ref{Shadow}, there exist a strictly increasing $C^1$ function $\theta:[0,l_nT]\to\mathbb{R}$ and a periodic point $p\in M$ such that
\begin{itemize}
   \item[(1)] $\theta(0)=0$ and $1-\varepsilon<\theta'(t)<1+\varepsilon$, for every $t\in[0,l_nT]$\text{;}
   	
   \item[(2)] $p$ is a periodic point with period $\theta(l_nT)$ {\rm:} $\overline{\varphi}_{\theta(l_nT)}(p)=p$\text{;}
   	
   \item[(3)] $d(\overline{\varphi}_t(y),\overline{\varphi}_{\theta(t)}(p))<\varepsilon\lvert X(\overline{\varphi}_t(y))\rvert$, for every $t\in[0,l_nT]$\text{.}
\end{itemize}    
   Moreover, by Proposition 4.4 in \cite{WYZ}, there exists a constant  $N=N(\eta,T)$ such that $$\lvert\theta(iT)-iT\rvert\leq Nd(y,\overline{\varphi}_{l_nT}(y)),\quad\text{for }i=1,2,\cdots,l_n.$$ Choose $l_n$ large enough so that $$d(y,\overline{\varphi}_{l_nT}(y))<\dfrac{\sigma_2}{2N}.$$ We shall denote $l_n$ simply by $l$ when no confusion can arise. Thus, we obtain a periodic orbit ${\rm Orb}(p)$ with period $\theta(lT)$. Then, we have a periodic measure $\mu_p$ supported on ${\rm Orb}(p)$, given by    $$\mu_p=\dfrac{1}{\theta(lT)}\int_{0}^{\theta(lT)}\delta_{\overline{\varphi}_t(p)}dt.$$
   
   Following the proof of Proposition \ref{TLLE}, we obtain $$\left\lvert\overline{\lambda}_i(p)-\overline{\lambda}_i\right\rvert<\epsilon,\text{ for each }i=d-\overline{d}_k,d-\overline{d}_k+1,\cdots,\overline{d}_k,$$ where $\overline{\lambda}_1(p)\leq\overline{\lambda}_2(p)\leq\cdots\leq\overline{\lambda}_{d-1}(p)$ are the Lyapunov exponents of the scaled linear Poincar\'{e} flow $\overline{\psi}^*_t$ with respect to the periodic invariant measure $\mu_p$. Since $$\overline{\lambda}_i=-\lambda_{d-i},~\overline{\lambda}_i(p)=-\lambda_{d-i}(p),\text{ for each } i=1,2,\cdots,d-1,$$ where $\lambda_1(p)\leq\lambda_2(p)\leq\cdots\leq\lambda_{d-1}(p)$ are the Lyapunov exponents of the linear Poincar\'{e} flow $\psi^*_t$ with respect to the periodic invariant measure $\mu_p$, we conclude that $$\lvert\lambda_i-\lambda_i(p)\rvert<\epsilon,\quad\text{for each }i=1,2,\cdots,d_1.$$ This completes the proof.    	 
\end{proof}

   To complete the proof of Theorem \ref{ThmB}, we recall some basic facts on exterior products. For more details on the Lyapunov exponents of automorphisms of exterior product spaces, we refer to \cite{A98,BV05}. Let $\varphi_t$ be the $C^1$ flow generated by a vector field $X\in\mathfrak{X}^1(M)$, and let $\mu$ be an ergodic hyperbolic invariant regular measure of $\varphi_t$. Let $\chi_1<\chi_2<\cdots<\chi_k$ denote the Lyapunov exponents of the scaled linear Poincar\'{e} flow $\psi^*_t$ with respect to $\mu$. The corresponding Oseledec splitting $$\mathcal{N}_\Gamma=E_1\oplus E_2\oplus\cdots\oplus E_k$$ is a dominated splitting for the scaled linear Poincar\'{e} flow $\psi^*_t$. 
   
   Let $\wedge^i\left(\mathbb{R}^d\right)$ denote the $i$-th exterior power of $\mathbb{R}^d$, which is a vector space of dimension $\binom{d}{i}$. A linear map $\mathcal{L}:\mathbb{R}^d\rightarrow\mathbb{R}^d$ naturally induces a linear map $\wedge^i\left(\mathcal{L}\right):\wedge^i\left(\mathbb{R}^d\right)\rightarrow\wedge^i\left(\mathbb{R}^d\right)$ defined by $$\wedge^i\left(\mathcal{L}\right)(v_1\wedge v_2\wedge\cdots\wedge v_i)=\mathcal{L}(v_1)\wedge\mathcal{L}(v_2)\wedge\cdots\wedge\mathcal{L}(v_i).$$ For a compact smooth Riemannian manifold $M$, we obtain the vector bundle $\wedge^i\left(TM\right)$ whose fiber at $x\in M$ is $\wedge^i\left(T_xM\right)$. In particular, for each $2\leq n\leq d-1$, we can construct a vector bundle $\wedge^n(\mathcal{N})$ of rank $\binom{d-1}{n}$ over the normal bundle $\mathcal{N}$. Its fiber at a non-singular point $x$ is $$\wedge^n(x)=\left\{v_{j_1}\wedge v_{j_2}\wedge\cdots\wedge v_{j_n}:~v_{j_i}\in\mathcal{N}_x,~1\leq i\leq n,~1\leq j_1<j_2<\cdots<j_n\leq d-1\right\}.$$ Let $$\wedge^n(\psi^*_t):\wedge^n(\mathcal{N})\rightarrow\wedge^n(\mathcal{N})$$ denote the $n$-exterior power of the scaled linear Poincar\'{e} flow $\psi^*_t$; that is, $$\wedge^n(\psi^*_t)(v_{j_1}\wedge v_{j_2}\wedge\cdots\wedge v_{j_n})=\psi^*_t(v_{j_1})\wedge\psi^*_t(v_{j_2})\wedge\cdots\wedge\psi^*_t(v_{j_n}).$$ The following lemma \ref{1} describes the Lyapunov exponents of this $n$-exterior power of the scaled linear Poincar\'{e} flow $\psi^*_t$.   
\begin{Lemma}{\rm\cite{A98,BV05}}\label{1}
   The Lyapunov exponents of $\wedge^n(\psi^*_t)$ at a point $x\in\Gamma$ are given by the sums $$\lambda_{j_1}+\lambda_{j_2}+\cdots+\lambda_{j_n},~~1\leq j_1<j_2<\cdots<j_n\leq d-1.$$
\end{Lemma}

   Let $\left\{e_1(x),e_2(x),\cdots,e_{d-1}(x)\right\}$ be a basis of $\mathcal{N}_x$ with the property that $$\text{dim}E_1(x)+\cdots+\text{dim}E_{j-1}(x)<i\leq \text{dim}E_1(x)+\cdots+\text{dim}E_j(x)\Rightarrow e_i(x)\in E_j(x).$$ The Oseledec subspace $\widehat{E}^{\wedge^n}_j$ of the $n$-exterior power of the scaled linear Poincar\'{e} flow $\wedge^n(\psi^*_t)$, corresponding to the Lyapunov exponent $\widehat{\lambda}_j=\lambda_{j_1}+\lambda_{j_2}+\cdots+\lambda_{j_n}$, is the subspace of $\wedge^n(\mathcal{N})$ spanned by the vectors $$e_{j_1}\wedge e_{j_2}\wedge\cdots\wedge e_{j_n},~1\leq j_1<j_2<\cdots<j_n\leq d-1.$$ Since the Oseledec splitting $$\mathcal{N}_\Gamma=E_1\oplus E_2\oplus\cdots\oplus E_k$$ is dominated, we can choose the basis $\left\{e_1(x),e_2(x),\cdots,e_{d-1}(x)\right\}$ of $\mathcal{N}_x$ so that for some $\gamma>0$, $$\measuredangle(e_{j_1}(x), e_{j_2}(x))>\gamma,\text{ for all $1\leq j_1\neq j_2\leq d-1$ and all $x\in\Gamma$}.$$ Identifying $v_{j_1}\wedge v_{j_2}\wedge\cdots\wedge v_{j_n}$ with the $n$-volume of the parallelepiped spanned by the vectors $v_{j_1}, v_{j_2},\cdots,v_{j_n}$ gives the standard norm on $\wedge^n(\mathcal{N})$. For computational purposes, however, it is convenient to introduce an equivalent norm  $\lVert\cdot\rVert_{\wedge^n}$ defined by $$\lVert v\rVert_{\wedge^n}=\max_{1\leq i\leq \binom{d-1}{n}}\left\{\lvert v_i\rvert\right\},\qquad v=\sum_{i=1}^{\binom{d-1}{n}}v_ie^{\wedge^n}_i,$$ where $\left\{e^{\wedge^n}_i\right\}_{i=1}^{\binom{d-1}{n}}$ is the basis of $\wedge^n(\mathcal{N})$ induced by $\left\{e_1(x),e_2(x),\cdots,e_{d-1}(x)\right\}$. 
   
   For each $n\in\left\{d_k+1,d_k+d_{k-1}+1,\cdots,d_k+\cdots+d_{s+2}+1\right\}$, let $\widehat{E}^{\wedge^n}_t$ be the Oseledec subspace corresponding to the largest Lyapunov exponents of $\wedge^n(\psi^*_t)$. Then, at a point $x\in M\setminus{\rm Sing}(X)$, the Oseledec splitting of $\wedge^n(\psi^*_t)$ is $$\widehat{E}^{\wedge^n}_1(x)\oplus\widehat{E}^{\wedge^n}_2(x)\oplus\cdots\oplus\widehat{E}^{\wedge^n}_{t-1}(x)\oplus\widehat{E}^{\wedge^n}_t(x).$$ From the definitions of dominated splitting and of the $n$-exterior power of the scaled linear Poincar\'{e} flow $\psi^*_t$, the next lemma \ref{2} follows directly; we therefore omit its proof. 
\begin{Lemma}\label{2}
   If the Oseledec splitting $$\mathcal{N}_\Gamma=E_1\oplus E_2\oplus\cdots\oplus E_k$$ of the scaled linear Poincar\'{e} flow $\psi^*_t$ with respect to an ergodic hyperbolic invariant regular measure $\mu$ is dominated, then
\begin{enumerate}
   \item[{\rm (1)}] for each $n\in\left\{d_k+1,d_k+d_{k-1}+1,\cdots,d_k+\cdots+d_{s+2}+1\right\}$, the Oseledec splitting $$\widehat{E}^{\wedge^n}_1\oplus\widehat{E}^{\wedge^n}_2\oplus\cdots\oplus\widehat{E}^{\wedge^n}_{t-1}\oplus\widehat{E}^{\wedge^n}_t$$ of $\wedge^n(\psi^*_t)$ is dominated and satisfies $$\left(\widehat{E}^{\wedge^n}_1\oplus\widehat{E}^{\wedge^n}_2\oplus\cdots\oplus\widehat{E}^{\wedge^n}_{t-1}\right)\prec\widehat{E}^{\wedge^n}_t\text{;}$$
   
   \item[{\rm (2)}] for each $n\in\left\{d_1+1,d_1+d_2+1,\cdots,d_1+\cdots+d_{s-1}+1\right\}$, the Oseledec splitting $$\widehat{E}^{\wedge^n}_1\oplus\widehat{E}^{\wedge^n}_2\oplus\cdots\oplus\widehat{E}^{\wedge^n}_{t-1}\oplus\widehat{E}^{\wedge^n}_t$$ of $\wedge^n(\psi^*_t)$ is dominated and satisfies $$\widehat{E}^{\wedge^n}_1\prec\left(\widehat{E}^{\wedge^n}_2\oplus\cdots\oplus\widehat{E}^{\wedge^n}_{t-1}\oplus\widehat{E}^{\wedge^n}_t\right)\text{.}$$       
\end{enumerate}   	
\end{Lemma}

   To complete the proof of Theorem \ref{ThmB}, we determine the Lyapunov exponents of $\wedge^n(\psi^*_t)$ by applying Propositions \ref{ITLLE} and \ref{ITSLE} below.
\begin{Proposition}\label{ITLLE}{\rm (The approximation of largest Lyapunov exponents)} 
   For $s+2\leq i\leq k$, set $n=n(i)=d_i+\cdots+d_k+1$. Then the flow $\wedge^n(\psi^*_t)$ admits $d_{i-1}$ largest Lyapunov exponents with respect to the ergodic hyperbolic invariant regular measure $\mu$. For every $\epsilon>0$, there exists a hyperbolic periodic point $p$ of $\varphi_t$ such that $$\left\lvert\widehat{\lambda}_j-\widehat{\lambda}_j(p)\right\rvert<\epsilon,\text{ for every }j=\tbinom{d-1}{n}-d_{i-1}+1,~\tbinom{d-1}{n}-d_{i-1}+2,\cdots,~\tbinom{d-1}{n},$$ where $\widehat{\lambda}_1(p)\leq\widehat{\lambda}_2(p)\leq\cdots\leq\widehat{\lambda}_{\tbinom{d-1}{n}}(p)$ are the Lyapunov exponents of $\wedge^n(\psi^*_t)$ with respect to the periodic measure $\mu_p$ supported on the periodic orbit ${\rm Orb}(p)$, and $\widehat{\lambda}_1\leq\widehat{\lambda}_2\leq\cdots\leq\widehat{\lambda}_{\tbinom{d-1}{n}}$ are the Lyapunov exponents of $\wedge^n(\psi^*_t)$ with respect to the ergodic hyperbolic invariant regular measure $\mu$.     	
\end{Proposition}

\begin{proof}
   For each integer $n\in[2,d-1]$, Lemma \ref{1} implies that the Lyapunov exponents of $\wedge^n(\psi^*_t)$ are all sums of the form $$\lambda_{j_1}+\lambda_{j_2}+\cdots+\lambda_{j_n},~~1\leq j_1<j_2<\cdots<j_n\leq d.$$ Consider the Oseledec splitting $\mathcal{N}_\Gamma=E_1\oplus E_2\oplus\cdots\oplus E_s\oplus E_{s+1}\oplus\cdots\oplus E_k$ of the scaled linear Poincar\'{e} flow $\psi^*_t$ with respect to the ergodic hyperbolic invariant measure $\mu$. The corresponding Lyapunov exponents satisfy $0<\chi_{s+1}<\chi_{s+2}<\cdots<\chi_k$; in particular, $$\lambda_j>0,\quad\text{ for every } j\geq d_1+d_2+\cdots+d_s+1.$$ Hence, for a given $n$, the largest Lyapunov exponent of $\wedge^n(\psi^*_t)$ is obtained by summing the Lyapunov exponents of the scaled linear Poincar\'{e} flow $\psi^*_t$ in reverse order, subject to the dimensional constraints imposed by the exterior power. Fix $i$ with $s+2\leq i\leq k$ and set $n=n(i)=d_i+\cdots+d_k+1$. The largest Lyapunov exponent of $\wedge^n(\psi^*_t)$ equals $$\lambda_{d-1}+\lambda_{d-2}+\cdots+\lambda_{d-n+1}+\chi_{i-1}.$$ Its associated Oseledec subspace is spanned by vectors of the form $e_{d-1}\wedge e_{d-2}\wedge\cdots\wedge e_{d-n+1}\wedge\tilde{e}$, where $\tilde{e}$ belongs to the set $\{e_j:e_j\subset E_{i-1}\}$. Because there are exactly $d_{i-1}=\text{dim}E_{i-1}$ independent choices for $\tilde{e}$, the multiplicity of this largest exponent is $d_{i-1}$. Consequently, $\wedge^n(\psi^*_t)$ possesses $d_{i-1}$ largest Lyapunov exponents. 
 
   Given $s+2\leq i\leq k$, we have a dominated splitting $$\mathcal{N}_\Gamma=\left(E_1\oplus E_2\oplus\cdots\oplus E_{i-1}\right)\prec\left(E_i\oplus\cdots\oplus E_k\right).$$ For brevity, we rewrite the dominated splitting as $$\mathcal{N}_\Gamma=E^{s,i}\oplus E^{u,i},\text{ where }E^{s,i}=E_1\oplus E_2\oplus\cdots\oplus E_{i-1},~E^{u,i}=E_i\oplus\cdots\oplus E_k.$$ Let $$\chi=\min_{1\leq j\leq k}\{\lvert\chi_j\rvert\}.$$ Since the Oseledec splitting $\mathcal{N}_\Gamma=E^s\oplus E^u$ is dominated, Lemma \ref{time} implies that for every $0<\epsilon/4\ll\chi$ and every $\delta\in(0,1/6d)$, there exists a positive number $L_i=L(\epsilon/4,\delta)$ such that
\begin{itemize}
   \item for every $T\geq L_i$, there exists a measurable set $\Gamma^{T,i}=\Gamma^{T,i}(\epsilon/4,\delta)\subset\Gamma$ with $\mu\left(\Gamma^{T,i}\right)\geq 1-\delta$;
 	
   \item there exists a natural number $N=N(T)$ such that for every integer $J\geq N$ and every $x\in\Gamma^{T,i}$, $$\prod_{j=0}^{J-1}\left\lVert\psi^*_T|_{E^{s,i}(\varphi_{jT}(x))}\right\rVert\leq e^{-(\chi-\epsilon/4)JT},~\prod_{j=0}^{J-1}m\left(\psi^*_T|_{E^{u,i}(\varphi_{jT}(x))}\right)\geq e^{(\chi-\epsilon/4)JT},$$ $$\text{ and }\dfrac{\left\lVert\psi^*_T|_{E^{s,i}}\right\rVert}{m\left(\psi^*_T|_{E^{u,i}}\right)}\leq e^{-T(\chi-\epsilon/4)}.$$     	  
\end{itemize}
   Fix an integer $T_0\geq L(\epsilon/4,\delta)$ and set $\eta_0=(\chi-\epsilon/4)T_0$. For every $C>0$, we define the {\bf Pesin block } $\Lambda^{T_0}_{\eta_0}(C)$ by
\begin{equation*}
\begin{aligned}
   \Lambda^{T_0}_{\eta_0}(C)=\Bigg\{x\in\Gamma:&\prod_{j=0}^{J-1}\left\lVert\psi^*_{T_0}|_{E^{s,i}\left(\varphi_{jT_0}(x)\right)}\right\rVert\leq Ce^{-J\eta_0},~\forall~J\geq 1,\\ &\prod_{j=0}^{J-1}m\left(\psi^*_{T_0}|_{E^{u,i}\left(\varphi_{jT_0}(x)\right)}\right)\geq C^{-1}e^{J\eta_0},~\forall~J\geq 1,~d(x,\text{Sing}(X))\geq\dfrac{1}{C}\Bigg\}.
\end{aligned}   
\end{equation*}
   According to \cite[Proposition 5.3]{WYZ}, $\Lambda^{T_0}_{\eta_0}(C)$ is compact and $$\mu\left(\Lambda^{T_0}_{\eta_0}(C)\right)\to\mu\left(\Gamma^{T_0,i}\right)\text{ as }C\to+\infty.$$ Hence we can choose a sufficiently large integer $C$ such that $\mu\left(\Lambda^{T_0}_{\eta_0}(C)\right)>1-2\delta>0$. For this fixed $C$, there exists a positive integer $j_0=j_0(C)$ satisfying $C<e^{\frac{j_0T_0\epsilon}{4}}$. Then for every $x\in\Lambda^{T_0}_{\eta_0}(C)$ and every $J\geq 1$, $$\prod_{j=0}^{J-1}\left\lVert\psi^*_{j_0T_0}|_{E^{s,i}\left(\varphi_{jj_0T_0}(x)\right)}\right\rVert\leq e^{-(\chi-\epsilon/2)Jj_0T_0},~\prod_{j=0}^{J-1}m\left(\psi^*_{j_0T_0}|_{E^{u,i}\left(\varphi_{jj_0T_0}(x)\right)}\right)\geq e^{(\chi-\epsilon/2)Jj_0T_0}.$$ Setting $T=j_0T_0$ and $\eta=(\chi-\epsilon/2)j_0T_0$, we consider the set
\begin{equation*}
\begin{aligned}
   \Lambda^{T}_{\eta}(C)=\Bigg\{x\in\Gamma:&\prod_{j=0}^{J-1}\left\lVert\psi^*_{T}|_{E^{s,i}\left(\varphi_{jT}(x)\right)}\right\rVert\leq e^{-J\eta},~\forall~J\geq 1,\\ &\prod_{j=0}^{J-1}m\left(\psi^*_{T}|_{E^{u,i}\left(\varphi_{jT}(x)\right)}\right)\geq e^{J\eta},~\forall~J\geq 1,~d(x,\text{Sing}(X))\geq\dfrac{1}{C}\Bigg\}.
\end{aligned}   
\end{equation*}	   
   
   Take $y\in\Lambda^{T}_{\eta}(C)\cap\text{supp}(\mu)$. By the Poincar\'{e} Recurrence Theorem, there exists an increasing sequence of integers $\{l_j\}$ such that $d(y,\varphi_{l_jT}(y))\rightarrow 0\text{ as } j\rightarrow+\infty$. Proceeding as in the proof of Proposition \ref{TLLE} and using the definition of $\wedge^n(\psi^*_t)$, Theorem \ref{Shadow} and Lemma \ref{2}, we can select suitable parameters $\varepsilon_i>0$, $\alpha\in(0,1/6d)$ and an index $l_{j(i)}$ to obtain a periodic point $p$ satisfying $$\left\lvert\widehat{\lambda}_j-\widehat{\lambda}_j(p)\right\rvert<\epsilon,\text{ for every }j=\tbinom{d-1}{n}-d_{i-1}+1,\tbinom{d-1}{n}-d_{i-1}+2,\cdots,\tbinom{d-1}{n},$$ where $\widehat{\lambda}_1(p)\leq\widehat{\lambda}_2(p)\leq\cdots\leq\widehat{\lambda}_{\binom{d-1}{n}}(p)$ are the Lyapunov exponents of $\wedge^n(\psi^*_t)$ with respect to the periodic measure $\mu_p$ supported on the periodic orbit ${\rm Orb}(p)$, and $\widehat{\lambda}_1\leq\widehat{\lambda}_2\leq\cdots\leq\widehat{\lambda}_{\binom{d-1}{n}}$ are the Lyapunov exponents of $\wedge^n(\psi^*_t)$ with respect to the ergodic hyperbolic invariant regular measure $\mu$.     	 
\end{proof}

\begin{Proposition}\label{ITSLE}{\rm (The approximation of smallest Lyapunov exponents)}
   For $1\leq i\leq s-1$, set $n=n(i)=d_1+\cdots+d_i+1$. Then the flow $\wedge^n(\psi^*_t)$ admits $d_{i+1}$ smallest Lyapunov exponents with respect to the ergodic hyperbolic invariant regular measure $\mu$. For every $\epsilon>0$, there exists a hyperbolic periodic point $p$ of $\varphi_t$ such that $$\left\lvert\widehat{\lambda}_j-\widehat{\lambda}_j(p)\right\rvert<\epsilon,\text{ for every }j=1,2,\cdots,d_{i+1},$$ where $\widehat{\lambda}_1(p)\leq\widehat{\lambda}_2(p)\leq\cdots\leq\widehat{\lambda}_{\binom{d-1}{n}}(p)$ are the Lyapunov exponents of $\wedge^n(\psi^*_t)$ with respect to the periodic measure $\mu_p$ supported on ${\rm Orb}(p)$, $\widehat{\lambda}_1\leq\widehat{\lambda}_2\leq\cdots\leq\widehat{\lambda}_{\binom{d-1}{n}}$ are the Lyapunov exponents of $\wedge^n(\psi^*_t)$ with respect to the ergodic hyperbolic invariant regular measure $\mu$.     	    	
\end{Proposition}

\begin{proof}
   For each integer $n\in[2,d-1]$, Lemma \ref{1} states that the Lyapunov exponents of $\wedge^n(\psi^*_t)$ are all sums of the form $$\lambda_{j_1}+\lambda_{j_2}+\cdots+\lambda_{j_n},~~1\leq j_1<j_2<\cdots<j_n\leq d-1.$$ Consider the Oseledec splitting $\mathcal{N}_\Gamma=E_1\oplus E_2\oplus\cdots\oplus E_s\oplus E_{s+1}\oplus\cdots\oplus E_k$ of the scaled linear Poincar\'{e} flow $\psi^*_t$ with respect to the ergodic hyperbolic invariant measure $\mu$. The corresponding Lyapunov exponents satisfy $\chi_1<\chi_2<\cdots<\chi_s<0$. Hence $$\lambda_j<0,\quad\text{ for every } j\leq d_1+d_2+\cdots+d_s.$$ Therefore, for a given $n$, the smallest Lyapunov exponent of $\wedge^n(\psi^*_t)$ is attained by summing the $n$ smallest individual exponents $\lambda_j$. More precisely, for each $1\leq i\leq s-1$, set $n=n(i)=d_1+\cdots+d_i+1$. Then the smallest Lyapunov exponent of $\wedge^n(\psi^*_t)$ equals $$\lambda_1+\lambda_2+\cdots+\lambda_{n-1}+\chi_{i+1}.$$ Geometrically, the Oseledec subspace corresponding to this smallest exponent is spanned by vectors of the form $e_1\wedge e_2\wedge\cdots\wedge e_{n-1}\wedge\tilde{e}$, where $\left\{e_1,\cdots,e_{n-1}\right\}$ is a basis for $E_1\oplus\cdots \oplus E_i$ and $\tilde{e}$ is any basis vector of $E_{i+1}$. Since $\tilde{e}$ can be any of the $d_{i+1}$ basis vectors of $E_{i+1}$, we obtain exactly $d_{i+1}$ linearly independent vectors spanning this Oseledec subspace. Consequently, the smallest Lyapunov exponent of $\wedge^{n(i)}(\psi^*_t)$ has multiplicity $d_{i+1}$.  
      
   Given $1\leq i\leq s-1$, we have a dominated splitting $$\mathcal{N}_\Gamma=\left(E_1\oplus E_2\oplus\cdots\oplus E_i\right)\prec\left(E_{i+1}\oplus\cdots\oplus E_k\right).$$ For brevity, we write $$\mathcal{N}_\Gamma=E^{s,i}\oplus E^{u,i},\text{ where }E^{s,i}=E_1\oplus E_2\oplus\cdots\oplus E_{i},~E^{u,i}=E_{i+1}\oplus\cdots\oplus E_k.$$ Let $$\chi=\min_{1\leq j\leq k}\{\lvert\chi_j\rvert\}.$$ Since the Oseledec splitting $\mathcal{N}_\Gamma=E^s\oplus E^u$ is dominated, Lemma \ref{time} implies that for every $0<\epsilon/4\ll\chi$ and every $\delta\in(0,1/6d)$, there exists a positive number $L_i=L(\epsilon/4,\delta)$ such that
\begin{itemize}
   \item for every $T\geq L_i$, there exists a measurable set $\Gamma^{T,i}=\Gamma^{T,i}(\epsilon/4,\delta)\subset\Gamma$ with $\mu\left(\Gamma^{T,i}\right)\geq 1-\delta$;
 	
   \item there exists a natural number $N=N(T)$ such that for every integer $J\geq N$ and every $x\in\Gamma^{T,i}$, $$\prod_{j=0}^{J-1}\left\lVert\psi^*_T|_{E^{s,i}\left(\varphi_{jT}(x)\right)}\right\rVert\leq e^{-(\chi-\epsilon/4)JT},~\prod_{j=0}^{J-1}m\left(\psi^*_T|_{E^{u,i}\left(\varphi_{jT}(x)\right)}\right)\geq e^{(\chi-\epsilon/4)JT},$$ $$\text{ and }\dfrac{\left\lVert\psi^*_T|_{E^{s,i}}\right\rVert}{m\left(\psi^*_T|_{E^{u,i}}\right)}\leq e^{-T(\chi-\epsilon/4)}.$$     	  
\end{itemize}
   Fix an integer $T_0\geq L(\epsilon/4,\delta)$ and let $\eta_0=(\chi-\epsilon/4)T_0$. For every $C>0$, we define the {\bf Pesin block } $\Lambda^{T_0}_{\eta_0}(C)$ by
\begin{equation*}
\begin{aligned}
   \Lambda^{T_0}_{\eta_0}(C)=\Bigg\{x\in\Gamma:&\prod_{j=0}^{J-1}\left\lVert\psi^*_{T_0}|_{E^{s,i}\left(\varphi_{jT_0}(x)\right)}\right\rVert\leq Ce^{-J\eta_0},~\forall~J\geq 1,\\ &\prod_{j=0}^{J-1}m\left(\psi^*_{T_0}|_{E^{u,i}\left(\varphi_{jT_0}(x)\right)}\right)\geq C^{-1}e^{J\eta_0},~\forall~J\geq 1,~d(x,\text{Sing}(X))\geq\dfrac{1}{C}\Bigg\}.
\end{aligned}   
\end{equation*}
   According to \cite[Proposition 5.3]{WYZ}, $\Lambda^{T_0}_{\eta_0}(C)$ is compact and $$\mu\left(\Lambda^{T_0}_{\eta_0}(C)\right)\rightarrow\mu\left(\Gamma^{T_0,i}\right)\text{ as }C\rightarrow+\infty.$$ Therefore, we can choose $C$ sufficiently large such that $\mu\left(\Lambda^{T_0}_{\eta_0}(C)\right)>1-2\delta>0$. For this fixed $C$, there exists a positive integer $j_0=j_0(C)$ such that $C<e^{\frac{j_0T_0\epsilon}{4}}$. Thus, for every $x\in\Lambda^{T_0}_{\eta_0}(C)$ and every $J\geq 1$, $$\prod_{j=0}^{J-1}\left\lVert\psi^*_{j_0T_0}|_{E^{s,i}\left(\varphi_{jj_0T_0}(x)\right)}\right\rVert\leq e^{-(\chi-\epsilon/2)Jj_0T_0},~\prod_{j=0}^{J-1}m\left(\psi^*_{j_0T_0}|_{E^{u,i}\left(\varphi_{jj_0T_0}(x)\right)}\right)\geq e^{(\chi-\epsilon/2)Jj_0T_0}.$$ Setting $T=j_0T_0$ and $\eta=(\chi-\epsilon/2)j_0T_0$, we consider the set
\begin{equation*}
\begin{aligned}
   \Lambda^{T}_{\eta}(C)=\Bigg\{x\in\Gamma:&\prod_{j=0}^{J-1}\left\lVert\psi^*_{T}|_{E^{s,i}\left(\varphi_{jT}(x)\right)}\right\rVert\leq e^{-J\eta},~\forall~J\geq 1,\\ &\prod_{j=0}^{J-1}m\left(\psi^*_{T}|_{E^{u,i}\left(\varphi_{jT}(x)\right)}\right)\geq e^{J\eta},~\forall~J\geq 1,~d(x,\text{Sing}(X))\geq\dfrac{1}{C}\Bigg\}.
\end{aligned}   
\end{equation*}	      
  
   Take $y\in\Lambda^{T}_{\eta}(C)\cap\text{supp}(\mu)$. By the Poincar\'{e} Recurrence Theorem, there exists an increasing sequence of integers $\{l_j\}$ such that $d(y,\varphi_{l_jT}(y))\to 0$ as $j\to+\infty$. Following the same strategy used in the proof of Proposition \ref{TSLE}, which relies on the definition of $\wedge^n(\psi^*_t)$, Theorem \ref{Shadow}, and Lemma \ref{2}, we can select appropriate parameters $\varepsilon_i>0$, $\alpha\in(0, 1/6d)$, and a sufficiently large index $l_{j(i)}$. Applying this construction to the orbit segment $\varphi_{[0,l_{j(i)}T]}(y)$, we obtain a periodic point $p$ such that $$\left\lvert\widehat{\lambda}_j-\widehat{\lambda}_j(p)\right\rvert<\epsilon,\text{ for every }j=\binom{d-1}{n}-d_{i-1}+1,\binom{d-1}{n}-d_{i-1}+2,\cdots,\binom{d-1}{n},$$ where $\widehat{\lambda}_1(p)\leq\widehat{\lambda}_2(p)\leq\cdots\leq\widehat{\lambda}_{\binom{d-1}{n}}(p)$ are the Lyapunov exponents of $\wedge^n(\psi^*_t)$ with respect to the periodic measure $\mu_p$ supported on the periodic orbit ${\rm Orb}(p)$, and $\widehat{\lambda}_1\leq\widehat{\lambda}_2\leq\cdots\leq\widehat{\lambda}_{\binom{d-1}{n}}$ are the Lyapunov exponents of $\wedge^n(\psi^*_t)$ with respect to the ergodic hyperbolic invariant regular measure $\mu$.     	          
\end{proof}

\begin{proof}[Proof of Theorem \ref{ThmB}]
   Let $\varepsilon>0$ be given. First, choose $\varepsilon_0 > 0$ sufficiently small as required by Propositions \ref{TLLE} and \ref{TSLE}. Define $$\mathcal{L}=\max\left\{L,\overline{L},L_1,L_2,\cdots,L_{s-1},L_{s+2},\cdots,L_k\right\},~\widetilde{\varepsilon}=\min\left\{\varepsilon_0,\varepsilon_1,\varepsilon_2,\cdots,\varepsilon_{s-1},\varepsilon_{s+2},\cdots,\varepsilon_k\right\},$$ where $L$ and $\overline{L}$ are the constants from Propositions \ref{TLLE} and \ref{TSLE}, respectively, and $L_i,~\varepsilon_i$ (for $i \neq s,~s+1$) are those given by Propositions \ref{ITLLE} and \ref{ITSLE}. Next, choose $\delta>0$ sufficiently small. Then the set $$\Lambda:=\Gamma^{\mathcal{L}}\cap\overline{\Gamma^{\mathcal{L}}}\cap\left(\bigcap_{i=1}^{s-1}\Gamma^{\mathcal{L},i}\right)\cap\left(\bigcap_{i=s+2}^{k}\Gamma^{\mathcal{L},i}\right)$$ has positive $\mu$-measure. Applying Propositions \ref{TLLE}, \ref{TSLE}, \ref{ITLLE} and \ref{ITSLE} with these parameters, there exists a hyperbolic periodic point $p$ such that $$\left\lvert\lambda_i-\lambda_i(p)\right\rvert<\varepsilon,\text{ for each }i=1,2,\cdots,d-1,$$ where $\lambda_1(p)\leq\lambda_2(p)\leq\cdots\leq\lambda_{d-1}(p)$ are the Lyapunov exponents of the scaled linear Poincar\'{e} $\psi^*_t$ with respect to the periodic $\mu_p$ supported on the periodic orbit $\text{Orb}(p)$.      
\end{proof}


\begin{thebibliography}{99}

\bibitem{ABS77} V. S. Afra\u{l}movi\u{c}, and V. V. Bykov, and L. P. Si\'{L}nikov, The origin and structure of the Lorenz attractor, {\it Dokl. Akad. Nauk SSSR}, {\bf 234} (1977), no. 2, 336--339.

\bibitem{A67} D. V. Anosov, Geodesic flows on closed Riemannian manifolds of negative curvature, {\it Trudy Mat. Inst. Steklov.}, {\bf 90} (1967), 3--210.

\bibitem{A98} L. Arnold, Random dynamical systems, {\it Springer Monographs in Mathematics}, Springer Monographs in Mathematics, Springer-Verlag, Berlin, 1998.

\bibitem{BV05} J. Bochi and M. Viana, The Lyapunov exponents of generic volume-preserving and symplectic maps, {\it Ann. of Math. (2)}, {\bf 161} (2005), no. 3, 1423--1485.

\bibitem{BDV05} C. Bonatti, and L. J. D\'{\i}az, and M. Viana, Dynamics beyond uniform hyperbolicity, {\it Encyclopaedia of Mathematical Sciences}, vol. 102, Springer-Verlag, Berlin, 2005, A global geometric and probabilistic perspective, Mathematical Physics, III.

\bibitem{B70} R. Bowen, Markov partitions for Axiom ${\rm A}$ diffeomorphisms, {\it Amer. J. Math.}, {\bf 92} (1970), 725--747.
   
\bibitem{BMW12} K. Burns, and H. Masur, and A. Wilkinson, The Weil-Petersson geodesic flow is ergodic, {\it Ann. of Math. (2)}, {\bf 175} (2012), no. 2, 835--908.
 
\bibitem{CLP89} S. N. Chow, and X. B. Lin, and K. J. Palmer, A shadowing lemma with applications to semilinear parabolic equations, {\it SIAM J. Math. Anal.}, {\bf 20} (1989), no. 3, 547--557. 
 
\bibitem{CS10} S. Crovisier, Birth of homoclinic intersections: a model for the central dynamics of partially hyperbolic systems, {\it Ann. of Math. (2)}, {\bf 172} (2010), no. 3, 1641--1677.

\bibitem{CY2017} S. Crovisier, and D. W. Yang, Homoclinic tangencies and singular hyperbolicity for three-dimensional vector fields, {\it preprint}, (2017), Arxiv:1702.05994.

\bibitem{Gan2002} S. B. Gan, A generalized shadowing lemma, {\it Discrete Contin. Dyn. Syst.}, {\bf 8} (2002), no. 3, 627--632.

\bibitem{GY} S. B. Gan, and D. W. Yang, Morse-Smale systems and horseshoes for three dimensional singular flows, {\it Ann. Sci. \'{E}c. Norm. Sup\'{e}r. (4)}, {\bf 51} (2018), no. 1, 39--112.

\bibitem{Guc76} J. Guckenheimer, A Strange, Strange Attractor, pp. 368–381, Springer New York, New York, NY, 1976.

\bibitem{GW79} J. Guckenheimer, and R. F. Williams, Structural stability of Lorenz attractors, {\it Inst. Hautes \'{E}tudes Sci. Publ. Math.}, {\bf 50} (1979), 59--72.

\bibitem{HW18} B. Han, and X. Wen, A shadowing lemma for quasi-hyperbolic strings of flows, {\it J. Differential Equations}, {\bf 264} (2018), no. 1, 1--29.

\bibitem{Ka} A. Katok, Lyapunov exponents, entropy and periodic orbits for diffeomorphisms, {\it Inst. Hautes \'Etudes Sci. Publ. Math.}, {\bf 51} (1980), 137--173.

\bibitem{LGW05} M. Li, and S. B. Gan, and L. Wen, Robustly transitive singular sets via approach of an extended linear {P}oincar\'{e} flow, {\it Discrete Contin. Dyn. Syst.}, {\bf 13} (2005), no. 2, 239--269.

\bibitem{LLL24} M. Li, and C. Liang, and X. Z. Liu, A closing lemma for non-uniformly hyperbolic singular flows, {\it Comm. Math. Phys.}, {\bf 405} (2024), no. 8, Paper No. 195, 35.

\bibitem{Lian} Z. Lian, and L. S. Young, Lyapunov exponents, periodic orbits, and horseshoes for semiflows on Hilbert spaces, {\it J. Amer. Math. Soc.}, {\bf 25} (2012), no. 3, 637--665.

\bibitem{LLS14} C. Liang, and G. Liao, and W. X. Sun, A note on approximation properties of the Oseledets splitting, {\it Proc. Amer. Math. Soc.}, {\bf 142} (2014), no. 11, 3825--3838.

\bibitem{LLS09} C. Liang, and G. Liu, and W. X. Sun, Approximation properties on invariant measure and Oseledec splitting in non-uniformly hyperbolic systems, {\it Trans. Amer. Math. Soc.}, {\bf 361} (2009), no. 3, 1543--1579. 

\bibitem{LST79A} S. T. Liao, A basic property of a certain class of differential systems, {\it Acta Math. Sinica}, {\bf 22} (1979), no. 3, 316--343.

\bibitem{LST79} S. T. Liao, An existence theorem for periodic orbits, {\it Beijing Daxue Xuebao}, (2005), no. 1, 1--20.

\bibitem{LST80} S. T. Liao, On the stability conjecture, {\it Chinese Annals of Mathematics Series A}, {\bf 1} (1980), no. 1, 8--30.

\bibitem{LST85} S. T. Liao, Some uniformity properties of ordinary differential systems and a generalization of an existence theorem for periodic orbits, {\it Beijing Daxue Xuebao}, (1985), no. 2, 1--19.

\bibitem{Liao89} S. T. Liao, On $(\eta,d)$-contractible orbits of vector ﬁelds, {\it Systems Sci. Math. Sci.}, {\bf 2} (1970), no. 3, 193--227.

\bibitem{Lor63} E. N. Lorenz, Deterministic nonperiodic flow, {\it J. Atmospheric Sci.}, {\bf 20} (1963), no. 2, 130--141.

\bibitem{LW2025} Y. S. Lu, and W. L. Wu, The Lyapunov exponents of hyperbolic measures for $C^1$ Star vector fields on three-dimensional manifolds, {\it preprint}, (2025), ArXiv:2507.23605.

\bibitem{M22} X. Ma, Existence of periodic orbits and horseshoes for semiflows on a separable Banach space, {\it Calc. Var. Partial Differential Equations}, {\bf 61} (2022), no. 6, Paper No. 217, 37.

\bibitem{MR} R. Ma\~{n}\'{e}, An ergodic closing lemma, {\it Ann. of Math. (2)}, {\bf 116} (1982), no. 3, 503--540.
	
\bibitem{OV} V. Oseledec, A multiplicative ergodic theorem. Characteristic Ljapunov, exponents of dynamical systems, {\it Trudy Moskov. Mat. Ob\v s\v C.}, {\bf 19} (1968), 179--210.

\bibitem{P99} S. Y. Pilyugin, Shadowing in dynamical systems, {\it Lecture Notes in Mathematics}, vol. 1706, Springer-Verlag, Berlin, 1999.

\bibitem{PM93} M. Pollicott, {\it Lectures on ergodic theory and Pesin theory on compact manifolds}, London Mathematical Society Lecture Note Series, vol. 180, Cambridge University Press, Cambridge, 1993.

\bibitem{PC84} C. C. Pugh, The $C^{1+\alpha}$ hypothesis in Pesin theory, {\it Inst. Hautes \'{E}tudes Sci. Publ. Math.}, (1984), no. 59, 143--161.

\bibitem{PS00} E. R. Pujals, and M. Sambarino, Homoclinic tangencies and hyperbolicity for surface diffeomorphisms, {\it Ann. of Math. (2)}, {\bf 151} (2000), no. 3, 961--1023.

\bibitem{SGW14} Y. Shi, and S. B. Gan, and L. Wen, On the singular-hyperbolicity of star flow, {\it J. Mod. Dyn.}, {\bf 8} (2014), no. 2, 191--219.

\bibitem{SW00} M. Shub, and A. Wilkinson, Pathological foliations and removable zero exponents, {\it Invent. Math.}, {\bf 139} (2000), no. 3, 495--508.

\bibitem{Sig70} K. Sigmund, Generic properties of invariant measures for Axiom {${\rm A}$} diffeomorphisms, {\it Invent. Math.}, {\bf 11} (1970), 99--109.

\bibitem{Si72} J. G. Sina\u{\i}, Gibbs measures in ergodic theory, {\it Uspehi Mat. Nauk}, {\bf 27} (1972), no. 4(166), 21--64.

\bibitem{WCZ21} J. Wang, and Y. L. Cao, and R. Zou, The approximation of uniform hyperbolicity for $C^1$ diffeomorphisms with hyperbolic measures, {\it J. Differential Equations}, {\bf 275} (2021), 359--390.

\bibitem{WW10} Z. Q. Wang, and W. X. Sun, Lyapunov exponents of hyperbolic measures and hyperbolic periodic orbits, {\it Trans. Amer. Math. Soc.}, {\bf 362} (2010), no. 8, 4267--4282.

\bibitem{WW19} X. Wen, and L. Wen, A rescaled expansiveness for flows, {\it Proceedings of Symposia in Pure Mathematics}, {\bf 371} (2019), no. 5, 3179--3207.

\bibitem{WYZ} W. L. Wu, and D. W. Yang, and Y. Zhang, On the growth rate of periodic orbits for vector fields, {\it Adv. Math.}, {\bf 346} (2019), 170--193.

\bibitem{ZS10} Y. H. Zhou, and W. X. Sun, The Lyapunov exponents of $C^1$ hyperbolic systems, {\it Sci. China Math.}, {\bf 53} (2010), no. 7, 1743--1752.

\end{thebibliography}

\noindent\text{Wanlou Wu}\\
School of Mathematics and Statistics\\
Jiangsu Normal University, Xuzhou, 221116, P.R. China\\
wuwanlou@163.com\\

\end{document}